\documentclass{scrartcl}

\scrollmode
\usepackage{amsmath, amsthm, amssymb,amsfonts,mathtools}
\usepackage{verbatim,geometry,graphicx,multirow,array}
\usepackage{algpseudocode}
\usepackage{enumerate}
\usepackage{xcolor}
\usepackage[normalem]{ulem}
\usepackage{xy,psfrag}
\usepackage{datetime}
\usepackage{eucal}
\usepackage{amsmath,amsthm,amssymb,amsfonts,xcolor}
\usepackage{graphicx, caption, subcaption, listings, float}
\usepackage{multirow,enumitem, hyperref, tabularx}
\usepackage{array,algorithm, algpseudocode}
\usepackage{makecell, setspace}
\usepackage{enumitem}

\usepackage{xcolor, mathtools}
\usepackage{tabularx,colortbl}

\newcommand{\R}{{\mathbb R}}

\newcommand{\entropy}{\mathrm{Ent}}
\newcommand{\Rn}{{\R^N}}

\newcommand{\dd}{\mathrm d}

\renewcommand{\ss}{\scriptstyle}

\def\sddots{\mathinner{\raise3pt\vbox{\hbox{$\ss .$}}
		\raise1.5pt\hbox{$\ss .$}\hbox{$\ss .$}}}
\let\hat\widehat
\theoremstyle{plain}
\newtheorem{thm}{Theorem}[section]
\newtheorem{lem}[thm]{Lemma}

\newtheorem{proposition}[thm]{Proposition}

\theoremstyle{definition}
\newtheorem{rem}[thm]{Remark}

\newtheorem{defn}[thm]{Definition}

\newcommand{\icol}[1]{% inline column vector
  \left(\begin{smallmatrix}#1\end{smallmatrix}\right)%
}

\newcommand{\change}[1]{\textcolor{blue}{#1}}
\graphicspath{{figures}}

\begin{document}

%\title{Regularized Semi-Discrete Optimal Transport with ODE Characterization}
\title{Characterizing and computing solutions to regularized semi-discrete optimal transport via an ordinary differential equation}

\author{Luca Nenna\thanks{Universit\'e Paris-Saclay, CNRS, Laboratoire de math\'ematiques d'Orsay, ParMA, Inria Saclay, 91405, Orsay, France. 
email: luca.nenna@universite-paris-saclay.fr},\thanks{Institut Universitaire de France, I.U.F.}\;
Daniyar Omarov \thanks{Department of Mathematical and Statistical Sciences, 632 CAB, University of Alberta, Edmonton, Alberta, Canada, T6G 2G1.
email: daniyar@ualberta.ca}
  and  Brendan Pass \thanks{Department of Mathematical and Statistical Sciences, 632 CAB, University of Alberta, Edmonton, Alberta, Canada, T6G 2G1.
email: pass@ualberta.ca}}
\maketitle

% \author[Nenna]{Luca Nenna}
% \address{Universit\'e Paris-Saclay, CNRS, Laboratoire de math\'ematiques d'Orsay, ParMA, Inria Saclay, 91405, Orsay, France}
% \email{luca.nenna@universite-paris-saclay.fr}
% \author[Omarov]{Daniyar Omarov}
% \address{Department of Mathematical and Statistical Sciences, 632 CAB, University of Alberta, Edmonton, Alberta, Canada, T6G 2G1}
% \email{daniyar@ualberta.ca}
% \author[Pass]{Brendan Pass}
% \address{Department of Mathematical and Statistical Sciences, 632 CAB, University of Alberta, Edmonton, Alberta, Canada, T6G 2G1}
% \email{pass@ualberta.ca}

\vskip\baselineskip\noindent

\begin{abstract}
This paper investigates the semi-discrete optimal transport (OT) problem with entropic regularization. We characterize the solution using a governing, well-posed ordinary differential equation (ODE).  This naturally yields an algorithm to solve the problem numerically, which we prove has desirable properties, notably including global strong convexity of a value function whose Hessian must be inverted in the numerical scheme. Extensive numerical experiments are conducted to validate our approach. We compare the solutions obtained using the ODE method with those derived from Newton’s method. Our results demonstrate that the proposed algorithm is competitive for problems involving the squared Euclidean distance and exhibits superior performance when applied to various powers of the Euclidean distance. In addition, it proves particularly effective in scenarios where the target points lie outside the support of the source measure.  Finally, we note that the ODE approach yields an estimate on the rate of convergence of the solution as the regularization parameter vanishes, for a generic cost function.
\end{abstract}

% \maketitle

% \pagestyle{myheadings}
% \thispagestyle{plain}
% \markboth{L.~Nenna, D.~Omarov, B.~Pass}{Regularized Semi-Discrete OT Problem}

\section{Introduction}\label{sec:intro}

The optimal transport (OT) problem, which was first proposed by Monge in 1781 \cite{Monge1781a} and subsequently relaxed by Kantorovich during the 1940-s \cite{Kantorovich1942a, kantorovich1948a}, involves identifying the most efficient way to transfer mass from one probability measure to another, all while minimizing a specified cost function. This profound problem is intricately connected to many areas of mathematics, including partial differential equations \cite{jordan1998variational, otto2001geometry} and statistics \cite{ramdas, bigotcazellespapadakis, weedberthlet}, and its applications are vast, extending into fields such as economics \cite{galichon2018optimal}, fluid mechanics \cite{BenamouBrenier, brenier1997homogenized}, image processing \cite{papadakis2015optimal}, and machine learning \cite{torres2021survey}. In recent decades, advancements in computational techniques, notably entropic regularization \cite{Cuturi2013a,benamou2015iterative,galichon2022cupid,chizat2018scaling}, have facilitated the practical implementation of OT in large-scale scenarios as well as the numerical resolution of many variational problems involving optimal transport terms \cite{peyre2015entropic,cuturi2014fast,chizat2018scaling,benamou2019entropy,blanchet2018computation,de2024variational,barilla2021mean}. %LN : I think we can  add some more papers} 

%{\color{red}BP: Need (a) reference(s) for this (implementation in gradient flows?)}. 

In this paper, we focus specifically on the semi-discrete optimal transport problem, that is 
\begin{equation}\label{eqn: unregularized OT}
   \max_{\gamma\in\Pi(\rho,\mu)}\int_{X \times Y} b(x,y)\dd \gamma\ , 
\end{equation}
where $b(x,y)$ is a cost function, %{\color{red} BP: I think it makes more sense to introduce the cost here.} 
the source measure $\rho(x)$ is absolutely continuous with respect to the Lebesgue measure, the target measure $\mu$ is supported on a finite set $Y$ and $\Pi(\rho,\mu)$ is the set of couplings having $\rho$ and $\mu$ as marginals. This particular variant has garnered increased interest recently due to its relevance in applications such as geometric optics \cite{decastro2014, DELEO2017123, doskolovich2019optimal} and mesh generation \cite{DU2002591, yang2018fast}.% {\color{orange} LN : I think we should add some extra references here} {\color{blue} DO: added}

Numerous studies have discussed various numerical algorithms relevant to this problem \cite{kitagawa, dieciomarov}; we refer the reader to \cite{ merigot:hal-02494446} and the references within for an overview on this topic. Motivated by some recent works in the discrete case \cite{nenna2022ode,hiew2024ordinary}, our focus here is on the entropic regularized version of the semi-discrete OT problem

\begin{equation}\label{eqn: regularized OT}
\max_{\gamma\in \Pi (\rho, \mu)} \int_{X\times Y}b(x,y)\dd\gamma -\varepsilon \entropy(\gamma \ |\ \rho \otimes \sigma)\ ,
\end{equation}
where $\sigma$ is the counting measure on $Y$ and $\entropy(\gamma\ |\ \rho\otimes\sigma)$ is the relative entropy with respect to the product measure.
In particular, we study the dual (semi-discrete) entropic problem
\begin{gather}
  \min_{\mathbf{v}\in\mathbb{R}^N} \int_{X}\varepsilon\log{\bigg{(}\sum_{k=1}^N e^{\frac{b(x,y_k) - v_k}{\varepsilon}}}\bigg{)}\dd\rho + \sum_{k=1}^N v_k\mu_k + \varepsilon \ , \ (\varepsilon>0)\ .
\end{gather}
We manage to equivalently  characterize the curve of solutions with respect to the regularization parameter $\varepsilon\mapsto v(\varepsilon)$  through a well-posed (even when the regularization parameter vanishes) ordinary differential equation.
This extends a result in the remarkable paper by Delalande \cite{delalande2022nearly} %(giving also new insights on the rate of convergence for a generic cost function, see \cite{altschuler2022asymptotics} for the quadratic case) 
which showed that a given curve of solutions to the entropic problem satisfy an ODE; but not the converse, that any solution of the corresponding Cauchy problem \emph{must} be in fact be a curve of solutions to the entropic OT problem.  Let us also emphasize that our work here applies to general cost functions, whereas \cite{delalande2022nearly} focuses exclusively on the quadratic cost $b(x,y) =-|x-y|^2$.  %Notable contributions to this area include works presented in \cite{, bercu2023stochastic}. 
% \red{BP: We should emphasize before the organizaional paragraph that we show the ODE is well posed (I don't think that's in Delande) and that we get uniform bounds near the solution.}
Moreover, we establish a uniform bound on the smallest eigenvalue of the Hessian of the dual functional when the curve $\varepsilon\mapsto v (\varepsilon)$ approaches the solution to the unregularized problem.  This bound seems to be a special feature of the semi-discrete setting, as we don't see a clear way to obtain similar estimates in fully discrete problems. It has important computational consequences, as it allows one to avoid numerical instabilities and therefore compute the solution to the fully unregularized problem \eqref{eqn: unregularized OT} by numerically solving the ODE characterizing solutions to \eqref{eqn: regularized OT} up to $\varepsilon =0$.  In addition, it allows us to estimate convergence rates of the solutions as $\varepsilon \rightarrow 0$, in the spirit of the results in \cite{altschuler2022asymptotics} and \cite{delalande2022nearly}, but for a generic cost function.

% In particular, we would characterize the regularized problem using ordinary differential equations (ODE); similar analyses for continuous densities can be found in \cite{hiew2024ordinary}, and for the multi-marginal OT problem in \cite{nenna2022ode}. The objective of this study is to perform a comprehensive numerical analysis of this variation of the problem. 
\paragraph{Overview of the paper}
The paper is organized as follows:  Section \ref{sec:theory} presents theoretical results associated with the problem statement, particularly highlighting the uniform boundedness of the Hessian matrix. Next, Section \ref{sec:ODE} establishes that  solutions can be characterized as solutions to an appropriate ODE. In Section \ref{sec:compEx}, we provide an extensive collection of computational examples for one-dimensional, two-dimensional, and three-dimensional examples, incorporating various cost functions. Additionally, we compare the solution derived from the ODE approach with the traditional Newton's method, observing that the proposed algorithm is competitive for problems involving the squared Euclidean distance but demonstrates superior performance when dealing with different powers of the Euclidean cost function and in scenarios where the target points are not contained within  the support of the source measure.

\section{Entropic optimal transport and the governing ODE}\label{sec:theory}
Consider a compact convex subset $ X \subset \mathbb{R}^n$ and a finite set $Y = \{y_1, y_2, \ldots, y_N\} \subset \mathbb{R}^n$ of $N$ points. Take two probability measures $\rho\in\mathcal P(X)$ and $\mu\in\mathcal P(Y)$ satisfying the following hypotheses:
\begin{itemize}
    \item[(H1)]  $\rho(x)$ is absolutely continuous with respect to the Lebesgue measure and bounded from above and below, that is $\exists$  $\overline{m},\underline{m}> 0$ such that  $0 < \underline{m} \leq \rho(x) \leq \overline{m} < \infty$;
    \item[(H2)] $\mu = \sum_{k=1}^N \mu_k \delta_{y_k}$ is a discrete probability measure on $Y$,  bounded from below by a positive constant, $\mu_k \geq \underline{\mu} > 0$, $\forall k$.

\end{itemize}
 Then, the entropic (semi-discrete) optimal transport (OT)  can be formulated as follows:
\begin{align}\label{RegProb1}
\max_{\gamma\in \Pi (\rho, \mu)} \int_{X\times Y}b(x,y)\dd\gamma -\varepsilon \entropy(\gamma \ |\ \rho \otimes \sigma)\ ,
\end{align}
%\red{BP:  don't we regularize with respect to uniform measure, $\mathcal L \otimes \sigma$?}
%\blue{LN: the way we wrote the dual implies that we have regularized w.r.t. $\rho\otimes\sigma$.}
where $\varepsilon\in [0, \change{\infty)}$ is a regularization parameter, $\Pi(\rho,\mu)$ is the set of probability measures on $X\times Y$ having $\rho$ and $\mu$ as marginals, $\sigma$ is the counting measure on $Y$, $b(x,y)$ is a cost function and $\entropy(\cdot|\rho \otimes \sigma)$ is the Boltzmann-Shannon relative entropy (or Kullback-Leibler divergence) w.r.t.\ the product measure $\rho \otimes \sigma$, defined for general probability measures $p,q$ as
\[
\entropy(p \,|\, q) = 
\begin{dcases*}
\displaystyle{\int_{\R^d} \eta \log(\eta)\, \dd q} & if $p = \eta q$,\\
+\infty & otherwise.
\end{dcases*}
\]
The fact that $q$ is a probability measure ensures that $\entropy(p \,|\, q) \geq 0$.
It is easy to show, see for instance \cite{bercu2023stochastic,genevay2016stochastic}, that \eqref{RegProb1} admits the following dual formulation :
\begin{gather}\label{Dual1}
  \min_{\mathbf{v}\in\mathbb{R}^N} \int_{X}\varepsilon\log{\bigg{(}\sum_{k=1}^N e^{\frac{b(x,y_k) - v_k}{\varepsilon}}}\bigg{)}\dd\rho + \sum_{k=1}^N v_k\mu_k + \varepsilon \ , \ \varepsilon>0\ .
\end{gather}
\begin{rem}
    Notice that the term $\varepsilon\log{\bigg{(}\sum_{k=1}^N e^{\frac{b(x,y_k) - v_k}{\varepsilon}}}\bigg{)}$ in \eqref{Dual1} is the so called b-soft-transform; roughly speaking this term converges to the usual b-transform of OT theory $v^{b}(x):=\sup_y b(x,y)-v(y)$  as $\varepsilon\to 0$.
\end{rem}
Before going into details let us reformulate problems \eqref{RegProb1} \eqref{Dual1} in a more convenient way:  consider the following change of variables: $t = 1-\varepsilon$ with $t\in[0,1]$ and  $b(x,y_i) = -tc(x,y_i)$. 
Then the primal problem \eqref{RegProb1} can be re-written as
\begin{align}\label{RegProb2}
    \boxed{\min_{\gamma\in \Pi (\rho,\mu)} t\int_{X\times Y}c(x,y)\dd\gamma + (1-t)\entropy(\gamma \ |\ \rho \otimes \sigma), \ t\in [0,1].\ }
\end{align}
Notice that \eqref{RegProb2} can be now seen as an interpolation between the case where the entropy is dominant, at $t=0$, for which the solution is explicit, $\gamma=\rho\otimes\sigma$, and the original unregularized semi-discrete optimal transport problem, at $t=1$.
It is straightforward to then derive the new dual problem which takes the form:
\begin{gather}\label{Dual2}
    \boxed{\min_{\mathbf{\psi}\in\mathbb{R}^N} \Phi(\mathbf{\psi},t):=\int_{X}(1-t)\log{\bigg{[}\sum_{k=1}^N e^{\frac{\psi_k - tc(x,y_k)}{1-t}}}\bigg{]}\dd\rho(x) - \sum_{k=1}^N\psi_k\mu_k - (1-t)\ .}
\end{gather}

From now we will focus on the regularized problems \eqref{RegProb2}, and \eqref{Dual2} and their unregularized counterparts 
\begin{equation*}
\min_{\gamma\in \Pi (\rho,\mu)}\int_X c(x,y)\dd\gamma(x,y),\quad 
\min_{\mathbf{\psi}\in\mathbb{R}^N} -\sum_{k=1}^N \int_{\text{Lag}_i(\psi)}(c(x,y)-\psi_i)\dd\rho-\sum_{k=1}^N\psi_k\mu_k\ , 
\end{equation*}
where $\text{Lag}_i(\psi)$ is the $i-$th Laguerre cell defined below.
 Moreover, we will make the following  assumption on the cost
\begin{itemize}
    \item[(H3)] the cost $c(x,y)$ is of class $\mathcal C^1$ and satisfies the twist condition, that is for all $x_0\in X$  $y\mapsto\nabla_x c(x_0,y)$ is injective.
\end{itemize}

%{\color{red}BP: So far it doesn't look like we are making any assumption on $c$.  We need some for this to be true (the intersection between Laguerre cells needs to be negligible).  I think we need an assumption on $c$ to justify the first step in the proof of Lemma 2.8 as well.}
%{\color{orange} LN : I think that for now twist is enough also to guarantee that the intersection between Laguerre cells is negligible}
\begin{defn}[Laguerre cell]
    Given a vector $\psi\in\mathbb{R}^N$, the Laguerre cell corresponding to a target point $y_i$ is:
    \begin{gather*}
        \text{Lag}_i(\psi) := \{x\in X \ | \ c(x,y_i) - \psi_i \leq c(x,y_k) - \psi_k,\ \forall k\neq i\}\ .
    \end{gather*}
    Furthermore, the measure of the Laguerre cell with respect to the density $\rho(x)$ supported on $X$ is given by $\rho(\text{Lag}_i(\psi)):=\int_{ \text{Lag}_i(\psi)}\dd\rho(x)$.
\end{defn}
%\begin{defn}[Smoothed Laguerre cell]
 %   Given a vector $\psi\in\mathbb{R}^N$, we define the entropic counterpart of the Laguerre cell above, that is the smoothed  Laguerre cell corresponding to a target point $y_i$:
  %  \begin{gather*}
   %     \text{RLag}^t_i(\psi) :=\dfrac{e^{\frac{\psi_i-tc(x,y_i)}{1-t}}}{\sum_{k=1}^Ne^{\frac{\psi_k-tc(x,y_k)}{1-t}}}\ .
   % \end{gather*}
%\end{defn}
%{\color{red}BP: Do we use this definition anywhere?}

\subsection{Preliminary results on semi-discrete entropic optimal transport and convexity of the objective function}

% let us state some preliminary results concerning the convexity  of \eqref{Dual1}. Indeed, notice that the functional to minimize in \eqref{Dual1} is the sum of a linear term $\langle \psi|\mu\rangle$ and another one given by
% \begin{equation}
%     \label{eq:KantFun}
%     \mathcal K(\psi):= \int_{X}\varepsilon\log{\bigg{(}\sum_{k=1}^N e^{\frac{b(x,y_k) - \psi_k}{\varepsilon}}}\bigg{)}\dd\rho+\varepsilon,
% \end{equation}
% thus we can focus only on $\mathcal K$ to deduce some strong convexity of the dual functional.

Before introducing the governing ODE in order to characterise \eqref{Dual2}, let us state some preliminary results concerning the entropic Kantorovich functional $\Phi(\psi,t)$. We will in particular focus on some convexity properties of the entropic Kantorovich functional in two different regimes: (1) the entropy term is dominant, that is the case when $t\in[0,t^*]$,  for some $t^*$ to be chosen; (2) the optimal transport term becomes stronger and the entropic solution $\psi(t)$ is close to the solution $\psi(1)$ of the unregularized problem $t\in[t^*,1]$. Notice that the existence of such a $t^*$ comes from the fact that we can easily show that $\psi(t)$ converges to $\psi(1)$ and so there exists a $t^*$ such that for all $t\geq t^*$ we have $\|\psi(t)-\psi(1)\|\leq \delta$ for some $\delta>0$, which will be selected later on to ensure appropriate bounds on the Hessian of $\Phi$.

Let us start by stating some general results and then see how we can improve them in the case where $t\in[t^*,1]$.

It is quite easy to show that a solution $\psi(t)$ to \eqref{Dual2} is actually bounded (notice that $X$ is compact and $c$ is a continuous cost function, %{\color{red}BP: we are now  suddenly assuming $c$ is continuous -- this assumption should be introduced explicitly, but is it enough for Lemma 2.8? }
 for more details see \cite{genevay2019entropy}[Chapter 3]) by $2\|c\|_\infty$ and it is well known that we can choose a solution $\psi(t)$ such that $\sum_{i=1}^N \psi_i =0$; it is therefore enough to show strong convexity of the functional $\Phi$ on the set 

\begin{equation}\label{eqn: domain}
U:=\{\psi\in\Rn\;|\;\|\psi\|\leq 2\|c\|_\infty,\;\psi\perp \mathbf{1}\}.
\end{equation}
In particular we obtain the following results{\color{red}.}
\begin{lem}\label{lem1}
    Let $\psi\in \mathbb{R}^N$ be a vector such that $\sum_{i=1}^N\psi_i = 0$. Also, let $\widehat{\mu}\in \mathbb{R}^N$ be a discrete probability vector with a lower bound $\underline{\widehat{\mu}}$, $\widehat{\mu}_i\geq \underline{\widehat{\mu}}>0$, $\forall i$. Denote by $\mathrm{Var}_{\widehat{\mu}}(\psi):=$ the variance of $\psi$ with respect to $\widehat{\mu}$. Then, it follows that $\mathrm{Var}_{\widehat{\mu}}(\psi)\geq \underline{\widehat{\mu}}||\psi||_2^2$.
\end{lem}
\begin{proof}
    Let $\bar{\psi}$ denote the expectation of $\psi$ with respect to the density $\widehat{\mu}$, $\bar{\psi} = \sum_{i=1}^N\widehat{\mu}_i\psi_i$.
    \begin{gather*}
        \text{Var}_{\widehat{\mu}}(\psi) = \sum_{i=1}^N\widehat{\mu}_i(\psi_i-\bar{\psi})^2\geq \sum_{i=1}^N\underline{\widehat{\mu}}(\psi_i-\bar{\psi})^2 = \underline{\widehat{\mu}}\sum_{i=1}^N(\psi_i-\bar{\psi})^2 \geq \underline{\widehat{\mu}}\sum_{i=1}^N\psi_i^2 = \underline{\widehat{\mu}}||\psi||_2^2\ ,
    \end{gather*}
    where 
    \begin{gather*}
        \sum_{i=1}^N(\psi_i-\bar{\psi})^2 = \sum_{i=1}^N(\psi_i^2 +\bar{\psi}^2 - 2\bar{\psi}\psi_i) = \sum_{i=1}^N\psi_i^2 +N\bar{\psi}^2 - 2\bar{\psi}\sum_{i=1}^N\psi_i = \sum_{i=1}^N\psi_i^2 +N\bar{\psi}^2 \geq \sum_{i=1}^N\psi_i^2\ .
    \end{gather*}
\end{proof}

\begin{lem}\label{lem2bis}
    Let $\rho$ a probability measure on the compact set $X$ satisfying (H1).  If $\widehat{\psi}\in U$, then $\exists \underline{\widehat{\mu}}>0$, such that $\widehat{\mu}_i\geq \underline{\widehat{\mu}}$, $\forall i$, where
    \begin{equation}\label{eqn: induced measure}
        \widehat{\mu}_i = \int_X\dfrac{e^{\frac{\widehat{\psi}_i-tc( x,y_i)}{1-t}}}{\sum_{k=1}^N e^{\frac{\widehat{\psi}_k-tc({ x},y_k)}{1-t}}}\dd\rho(x)\ , t\in [0,1)  .
    \end{equation}
\end{lem}
\begin{proof}
    By using the bounds on $\widehat{\psi}$, $c$ and $\rho$ we get
    \[ \widehat{\mu}_i\geq \dfrac{\underline m}{N{e^{2\frac{1+t}{1-t}}\|c\|_\infty}}:= \underline{\widehat{\mu}}\ .\]
\end{proof}
%{\color{red}BP: Is there a factor of $N$ missing?}
%{\color{orange} LN : I think you are right}
\begin{rem}
Notice that 
    \begin{equation}\label{eqn: gradient}
        \{\nabla_\psi \Phi(\widehat\psi,t)\}_i =  \int_X\dfrac{e^{\frac{\widehat{\psi}_i-tc( x,y_i)}{1-t}}}{\sum_{k=1}^N e^{\frac{\widehat{\psi}_k-tc( x,y_k)}{1-t}}}\dd\rho(x)- \mu_i\ , t\in [0,1) \ .
    \end{equation}
    %{\color{red}BP: I guess we mean $\mu_i$ here?}
\end{rem}
 In what follows, we will denote by $\nabla_{\psi,\psi}^{2,\perp}\Phi$ the Hessian of $\Phi$ with respect to $\psi$, restricted to the tangent space of $U$, $T_\psi U = \{v \in \mathbb{R}^N: \sum_{n=1}^kv_k=0\}$.  It is well known that the full Hessian $\nabla_{\psi,\psi}^{2}\Phi$ of $\Phi$ is not invertible, since $\Phi$ is constant along the direction $\mathbf{1}$, but $\nabla_{\psi,\psi}^{2, \perp}\Phi$ is invertible, and, as we show below, has uniform lower bounds in appropriate regions. 
\begin{thm}[Strong convexity of $\Phi$]
\label{thm:convexity}
     Let $\rho\in\mathcal P(X)$ and $\mu\in\mathcal P(Y)$ satisfy hypotheses (H1) and (H2), respectively. If $\widehat\psi\in U$ then there exists $\widehat C =\widehat C(t)>0$ such that  $\nabla_{\psi,\psi}^{2,\perp} \Phi(\widehat\psi,t) \succeq\widehat{C}\mathrm{Id_{N-1}} > 0$ for $t\in[0,1)$,
     where
    \begin{equation}\label{eqn: hessian formula}   
    \nabla_{\psi,\psi}^2 \Phi(\widehat\psi,t)=\frac{1}{1-t}\big(\mathrm{diag}(\nabla \Phi(\widehat\psi,t)-\nabla \Phi(\widehat\psi,t)\otimes\nabla \Phi(\widehat\psi,t)\big).
    \end{equation}
    %and 
    %\begin{gather*}
     %   \{\nabla_\psi\Phi(\widehat\psi,t)\}_i = \int_X\frac{e^{\frac{\widehat{\psi}_i-tc(x,y_i)}{1-t}}}{\sum_{k=1}^N e^{\frac{\widehat{\psi}_k-tc(x,y_k)}{1-t}}}\dd\rho(x)- \widehat{\mu}_i\ .
    %\end{gather*}
   % {\color{red}BP: Here and in what follows, we really mean the Hessian restricted to orthogonal complement to $\mathbf{1}$.  Maybe we should make a global remark about this, or define new notation for it.}
   %{\color{orange} LN : I do not know if it would be better but I would probably use $\nabla^2_\perp$ what do you think?}
\end{thm}
\begin{proof}
    Take a $\widehat\psi\in U$ we easily get that:
    \begin{gather*}
       \langle v, \nabla_{\psi,\psi}^2 \Phi(\widehat{\psi},t)v\rangle\geq \frac{1}{\tilde C}\text{Var}_{\widehat{\mu}}(v)\ \forall v\in\Rn  ,
    \end{gather*}
    where $\tilde C = e^{2||c||_\infty\text{diam}(X)}\frac{\overline{m}}{\underline{m}}+(1-t)>0$, $\forall t\in[0,1)$ and $\widehat{\mu}$ is defined as in Lemma \ref{lem2bis}. Furthermore, using Lemma \ref{lem1}, we obtain, for $v \in T_\psi U$,
    \begin{gather*}
        \text{Var}_{\widehat{\mu}}(v)\geq \underline{\widehat{\mu}}||v||_2^2\implies \langle v, \nabla_{\psi,\psi}^2 \Phi(\widehat{\psi},t)v\rangle \geq \frac{1}{\tilde C}\underline{\widehat{\mu}}||v||_2^2\ .
    \end{gather*}
    Finally, using Lemma \ref{lem2bis}, one derives the existence and bound on $\underline{\widehat{\mu}}$ and thus the lower bound for the smallest eigenvalue of $\nabla_{\psi,\psi}^{2,\perp}\Phi$:
    \begin{gather*}
        \lambda_{\min}\{\nabla_{\psi,\psi}^{2,\perp} \Phi(\widehat{\psi},t)\}\geq \frac{1}{C} >0,\; \forall t\in[0,1)\ . 
    \end{gather*}
    As a result, we get that $C = \frac{\underline m}{\tilde Ce^{2\frac{1+t}{1-t}\|c\|_\infty}}$, and so the Hessian matrix $\nabla_{\psi,\psi}^{2} \Phi(\widehat{\psi},t)$ is positive semi-definite with a simple zero eigenvalue and corresponding eigenvector $\mathbf{1}$. 
  % {\color{red}BP: These last two sentences together are very bizarre!  We'll need to clean this up once we decide how to describe the Hessian.}
\end{proof}

Notice that the result above provides strong convexity only on the interval $[0,1)$ since the lower bound degenerates as $t\to 1$. In order to establish strong convexity for all $t \in [0,1]$, it is necessary to undertake a more detailed analysis of the case in which the entropic solution closely approximates the solution of the unregularized problem. In particular, we focus on the regime $t \in [t^*,1]$.

% In order to get strong convexity for all $t\in [0,1]$ we should better study the case in which the entropic solution is close to the solution of the unregularized problem, namely we focus on the case in which $t\in[t^*,1]$. 
\subsection{Strong convexity on $[t^*,1]$}
Firstly, we state some results when we consider a set of vectors $\psi$ which are small perturbations of a vector $\psi^0$ such that $\mu_i = \rho(\text{Lag}_i(\psi^0))$. The main results here is  that in this region we can obtain a lower bound on the smallest non-zero eigenvalue of $\nabla^2_{\psi,\psi}\Phi(\widehat{\psi},t)$ which does not depend on the regularization parameter $t$.

%again the strong convexity of $\Phi$ but in the case where   the lower bound for the smallest eigenvalue of $\nabla^2_{\psi,\psi}\Phi(\widehat{\psi},t)$ does not depend on the regularization parameter. 
\begin{lem}\label{lem2}
    Let $\mu\in\mathbb{R}^N_+$ be a discrete probability vector with a lower bound $\underline{\mu}$, $\mu_i\geq \underline{\mu}>0$, $\forall i$ and $c$ a cost function satisfying (H3). In addition, let $\psi^0\in\mathbb{R}^N$ be a vector such that $\mu_i = \rho(\text{Lag}_i(\psi^0))$, $\forall i$. 
     Then there exists $\delta, \underline{\widehat{\mu}}>0$ such that
   % If $\widehat{\psi}$ is a small perturbation of $\psi^0$,
   if $||\widehat{\psi} - \psi^0||_2^2\leq \delta$, then %\exists \underline{\widehat{\mu}}>0$, such that 
   $\widehat{\mu}_i\geq \underline{\widehat{\mu}}$, $\forall i$, where $\widehat{\mu}_i$ is given by \eqref{eqn: induced measure}.
    %\begin{gather*}
     %    \widehat{\mu}_i = \int_X\frac{e^{\frac{\widehat{\psi}_i-tc(x,y_i)}{1-t}}}{\sum_{k=1}^N e^{\frac{\widehat{\psi}_k-tc(x,y_k)}{1-t}}}\dd\rho(x) , t\in[0,1] \ .
    %\end{gather*}}
\end{lem}
\begin{proof}
  By Proposition 38 in \cite{merigot:hal-02494446} , the mapping $\hat \psi \mapsto \rho(\text{Lag}_i(\widehat{\psi}))$ is continuous, and so for $\delta$ sufficiently small, $||\hat \psi -\psi^0|| \le\delta$ implies   
   
  % Given that $\widehat{\psi}$ is a small perturbation of $\psi^0$, it follows that: 
  $$\rho(\text{Lag}_i(\widehat{\psi}))\geq \frac{1}{2}\rho(\text{Lag}_i(\psi^0)) \geq \frac{1}{2} \underline{\mu}\ .$$

   % {\color{red}BP: We need some assumption on $c$ to ensure this.  In particular, I think this is where the $\delta$ depencence arises, which one referee asks about, and perhaps this should be noted.}{\color{orange} LN : you are right if the cost is not twisted we are in real troubles...I think what we want, some smotthness of the Laguerre cells, is exactly propostion 38 in \cite{merigot:hal-02494446}}
    Next, using that $\text{Lag}_i(\widehat{\psi}_i)\subset X$, $\forall i$, and $X = \cup_{i=1}^N\text{Lag}_i(\widehat{\psi}_i)$, we get that:
    \begin{gather*}
        \widehat{\mu}_i = \int_X\frac{e^{\frac{\widehat{\psi}_i-tc(x,y_i)}{1-t}}}{\sum_{k=1}^N e^{\frac{\widehat{\psi}_k-tc(x,y_k)}{1-t}}}\dd\rho(x)\geq \int_{\text{Lag}_i(\widehat{\psi})}\frac{e^{\frac{\widehat{\psi}_i-tc(x,y_i)}{1-t}}}{\sum_{k=1}^N e^{\frac{\widehat{\psi}_k-tc(x,y_k)}{1-t}}}\dd\rho(x)\ .
    \end{gather*}
    From the definition of Laguerre cell it follows: 
    \[x\in \text{Lag}_i(\widehat{\psi}) \iff c(x,y_i) - \widehat{\psi_i}\leq c(x,y_k) - \widehat{\psi_k} \iff c(x,y_i)-c(x,y_k) + \widehat{\psi_k} - \widehat{\psi_i} \leq 0\ , \ \forall k\neq i.\]
    Thus, we get:
    \begin{gather*}
    \begin{split}
        \widehat{\mu}_i &\geq \int_{\text{Lag}_i(\widehat{\psi})}\frac{1}{\sum_{k=1}^N e^{\frac{tc(x,y_i)-tc(x,y_k)+\widehat{\psi}_k - \widehat{\psi}_i}{1-t}}}\dd\rho(x)\\
        & \geq \int_{\text{Lag}_i(\widehat{\psi})}\frac{1}{\sum_{k=1}^N e^{c(x,y_k)-c(x,y_i)}}\dd\rho(x)\\
        & \geq \int_{\text{Lag}_i(\widehat{\psi})}\frac{1}{\sum_{k=1}^N e^{2\|c\|_\infty} }\dd\rho(x) = \frac{1}{Ne^{2\|c\|_\infty}}\rho(\text{Lag}_i(\widehat{\psi})) \geq \frac{1}{2Ne^{2\|c\|_\infty}}\underline{\mu}\ .
    \end{split}
    \end{gather*}
    Finally, it follows that $\widehat{\mu}_i \geq \underline{\widehat{\mu}}$, $\forall i$, where $\underline{\widehat{\mu}}=\frac{1}{2Ne^{2\|c\|_\infty}}\underline{\mu}>0$.
\end{proof}
\begin{thm}\label{thm: uniform convexity independent of t}
 Let $\rho\in\mathcal P(X)$ and $\mu\in\mathcal P(Y)$ satisfy hypotheses (H1) and (H2), respectively, and $c$ be a cost function satisfying (H3). In addition, let $\psi^0\in\mathbb{R}^N$ be a vector such that $\mu_i = \rho(\text{Lag}_i(\psi^0))$, $\forall i$. Then, there exists a constant $\widehat{C}>0$, which is independent of $t$,  and a $\delta >0$ such that if %$\widehat{\psi}$ is a small perturbation of $\psi^0$, 
 $||\widehat{\psi} - \psi^0||_2^2\leq \delta$, and $\widehat{\psi}\perp \mathbf{1}$, then $\lambda_{\min}\{\nabla_{\psi,\psi}^{2, \perp} \Phi(\widehat{\psi},t)\}\geq \widehat{C} > 0$.
\end{thm}

\begin{proof}
    Take a $\widehat\psi$ satisfying the hypothesis of the theorem, similarly to Theorem 3.2 of \cite{delalande2022nearly}, we easily get that:
    \begin{gather*}
        \langle v,\nabla_{\psi,\psi}^2 \Phi(\widehat\psi,t)v\rangle\geq \frac{1}{C}\text{Var}_{\widehat{\mu}}(v)\ \forall v\in\Rn  ,
    \end{gather*}
    where $C = e^{L_c\text{diam}(X)}\frac{\overline{m}}{\underline{m}}+1>0$, $L_c$ is the Lipschitz constant of the cost  and $\widehat{\mu}$  is defined as in Lemma \ref{lem2}. Furthermore, using Lemma \ref{lem1}, we obtain that 
    \begin{gather*}
        \text{Var}_{\widehat{\mu}}(v)\geq \underline{\widehat{\mu}}||v||_2^2\implies \langle v,\nabla_{\psi,\psi}^{2, \perp} \Phi(\widehat\psi,t)v\rangle\geq \frac{1}{C}\underline{\widehat{\mu}}||v||_2^2\ .
    \end{gather*}
    Finally, using Lemma \ref{lem2}, one derives the existence and bound on $\underline{\widehat{\mu}}$:
    \begin{gather*}
        \underline{\widehat{\mu}} = \frac{1}{2N}\underline{\mu} \implies \lambda_{\min}\{\nabla_{\psi,\psi}^{2,\perp} \Phi(\widehat\psi,t)\}\geq \frac{1}{C}\frac{1}{2Ne^{2\|c\|_\infty}}\underline{\mu} >0\ . 
    \end{gather*}
    As a result, we get that $\widehat{C} = \frac{1}{2NCe^{2\|c\|_\infty}}\underline{\mu}$. %{\color{red}BP: does the next phrase really add anything?} {\color{orange} LN : Nope.} \sout{and the Hessian matrix is positive semi-definite with simple zero eigenvalue and corresponding eigenvector $\mathbf{1}$}.
\end{proof}

\section{An ODE characterization of semi-discrete entropic optimal transport}
\label{sec:ODE}
Notice now that since the functional is convex any minimizer of \eqref{Dual2} is equivalently a solution to the equation $\nabla_\psi \Phi(\mathbf{\psi}(t), t) = 0$. Thanks to the regularity of $\Phi$ one can now differentiate with respect to $t$ and obtain the following governing ODE:
\begin{gather}\label{ODE}
    \nabla_{\psi,\psi}^2\Phi(\mathbf{\psi}(t),t)\mathbf{\psi}'(t) + \frac{\partial}{\partial t}\nabla_{\psi}\Phi(\mathbf{\psi}(t),t) = 0,\ t\in[0,1]\ ,
\end{gather}
where the gradient vector of $\Phi(\mathbf{\psi},t)$ from \eqref{Dual2} is given by \eqref{eqn: gradient}.
%\begin{gather}\label{GradPhi1}
 %   \frac{\partial \Phi(\mathbf{\psi},t)}{\partial \psi_j} := \int_X\frac{e^{\frac{\psi_j - tc(x,y_j)}{1-t}}}{\sum_{k=1}^N e^{\frac{\psi_k - tc(x,y_k)}{1-t}}}\dd\rho(x)\ -\mu_j, \ j=1,2,\dots, N\ .
%\end{gather}
%{\textcolor{red}{ DO: Equation \eqref{GradPhi1} appears to be inconsistent with the previous notation; in particular, with the equation in Remark 2.6}}

By using a direct calculation, we get the derivative vector:
\begin{gather}
\label{eq:der}
    \frac{\partial}{\partial t} \bigg{[}\frac{\partial \Phi(\mathbf{\psi},t)}{\partial \psi_j}\bigg{]} = \int_X\frac{\sum_{k=1, k\neq j}^N A_k(x,1)e^{\frac{A_k(x,t)]}{1-t}}}{\bigg{[}(1-t)\bigg{(}1 + \sum_{k=1,k\neq j}^N e^{\frac{A_k(x,t)}{1-t}}\bigg{)}\bigg{]}^2}\dd\rho(x)\ ,\ j=1,2,\dots, N\ ,
\end{gather}
where $A_k(x,t) = \psi_k - \psi_j + tc(x,y_j) - tc(x,y_k)$. Similarly, one can write the Hessian matrix in \eqref{eqn: hessian formula} as:
\begin{gather*}
    \nabla^2_{\mathbf{\psi},\mathbf{\psi}}\Phi(\mathbf{\psi}(t),t) = \frac{1}{1-t}\mathbb{E}_{\rho}[\pi(\mathbf{\psi})\pi(\mathbf{\psi})^T - \text{diag}(\pi(\mathbf{\psi}))]\ ,
\end{gather*}
where $\mathbb{E}_{\rho}[f(x)] := \int_Xf(x)\dd\rho(x)$ and 
\begin{gather*}
    \pi(\mathbf{\psi})_{j} = \frac{e^{\frac{\psi_j - tc(x,y_j)}{1-t}}}{\sum_{k=1}^N e^{\frac{\psi_k - tc(x,y_k)}{1-t}}}\ , \ j=1,2,\dots, N\ .
\end{gather*}
%Hence we get the following Cauchy problem.
Since $\psi(0)$ can be computed in closed form, it is clear that the solution $\psi(t)$ to \eqref{Dual2} solves the Cauchy problem \eqref{cauchy}; the following Theorem asserts the converse, that any solution to \eqref{cauchy} must be an optimal trajectory of solutions to \eqref{Dual2}.

%d{BP: One option would to to state that the solution is characterized by the ODE on $[0,1)$, and that of course if converges to the unregularized solution as $t \rightarrow 1$, so that solving the ODE and taking the limit really does give the solution.  An advantage of this statement is that I don't think we need structural assumptions ons $c$.}

%\red{If we want to say that the ODE is satisfied up to $t=1$, I think we need assumptions (to even say the Hessian exists at $t=1$).  We also likely need to be able to say that the RHS converges to $0$, which should hold as in \cite{delalande2022nearly}, again maybe with some assumptions on $c$ and $\rho$.} 

%\red{Maybe we should make both statements, as the first requires less assumptions.}\\
%\blue{LN: agree think we can do both statement}
\begin{thm}
     Let $\psi(t)$ be a solution to \eqref{Dual2} for $t\in[0,1)$. Then the trajectory $t\mapsto \psi(t)$ is smooth and is characterized on $t\in[0,1)$ as the unique solution $\psi(t) \in U$ to the Cauchy problem
\begin{equation}
 \label{cauchy}
\boxed{ \begin{cases}
&\nabla_{\psi,\psi}^2\Phi(\mathbf{\psi}(t),t)\mathbf{\psi}'(t) + \frac{\partial}{\partial t}\nabla_\psi\Phi(\mathbf{\psi}(t),t) = 0,\ t\in[0,1)\ ,\\
 & \mathbf{\psi}(0)=\log{\mu} - \frac{1}{N}\sum_{k=1}^N\log{\mu_k}\ .
\end{cases}  } 
\end{equation} 

\end{thm}

\begin{rem}
The theorem implies that one can solve the Cauchy problem \eqref{cauchy} uniquely to obtain $\psi(t)$ for $t \in [0,1)$.  As it is well known that $\lim_{t \rightarrow 1} \psi(t) =\psi(1)$, one can then take the limit to obtain $\psi(1)$ and, consequently, one can recover the solution to the unregularized problem from the ODE.  In practice, one might worry that numerical instabilities could potentially arise from degeneracies of the Hessian as $t$ approaches $1$; Theorem \ref{thm: uniform convexity independent of t} ensures that this is not the case.   

Under stronger conditions, one can actually go slightly further and show that the ODE is satisfied up to $t =1$,  see Proposition \ref{prop: ODE up to 1} below. 
\end{rem}

\begin{proof}
For any fixed $\bar t <1$, we prove existence  and uniqueness of a solution to the Cauchy problem on $[0,\bar t]$ by applying the Cauchy-Lipschitz theorem on $[0,\bar t] \times U$, where $U:=\{\psi\in\Rn\;|\;\|\psi\|\leq 2\|c\|_\infty,\;\psi\perp \mathbf{1}\}$.  Since $\Phi$ is smooth, it has  bounded derivative on this set, and, as noted above, it is known that the optimal $\psi$ remains in $U$. Since the restricted Hessian $\nabla_{\psi,\psi}^{2, \perp}\Phi(\mathbf{\psi},t)$ is uniformly positive definite on $[0,\bar t] \times U$ by Theorem \ref{thm:convexity}, one can rewrite the ODE as:
$$\mathbf{\psi}'(t) =F(\psi(t),t):=-\big[\nabla_{\psi,\psi}^{2, \perp}\Phi(\mathbf{\psi}(t),t)\big]^{-1}\big( \frac{\partial}{\partial t}\nabla_\psi\Phi(\mathbf{\psi}(t),t)\big)\ ,
$$
where the function $F$ is Lipschitz on $[0,\bar t] \times U$.  The Cauchy-Lipschitz theorem then yields well-posedness on $[0,\bar t]$; as $\bar t<1$ is arbitrary, we get existence and uniqueness on $[0,1)$. 

\end{proof}
To ensure that $\psi(t)$ solves the ODE in \eqref{cauchy} up to $t=1$, we require stronger conditions on $c, X,Y$ and $\rho$, which we define in the following. Before doing this we need two more definitions
%\begin{defn}[Twisted cost]
%We say that $c$ is \emph{twisted} if for each $x \in X$ the mapping $y \mapsto \nabla_xc(x,y)$ is injective on $Y.$  
%\end{defn}
%{\color{orange} LN : we probably have to put this when we do the assumption on $c$}
\begin{defn}[$Y$ generic with respect to $c$]
We say that $Y$ is \emph{generic with respect to $c$} if for all distinct, $y_0,y_1,y_2 \in Y$ the intersection of any level sets of $x \mapsto c(x,y_0)-c(x,y_1)$  and $x \mapsto c(x,y_0)-c(x,y_2)$ has $(n-1)$ -dimensional Hausdorff measure $0$. 
\end{defn}
\begin{defn}[$Y$ generic with respect to $\partial X$]
 We say $Y$ is \emph{generic with respect to $\partial X$} if for any $y_0,y_1 \in Y$, the intersection 
of any level set of $x \mapsto c(x,y_0)-c(x,y_1)$ with $\partial X$ has $(n-1)$ -dimensional Hausdorff measure $0$.
\end{defn}
\begin{itemize}
     \item[(H4)] The measure  $\rho$ satisfies the assumption (H1) with an $\alpha-$Hölder continuous density.
     \item[(H5)] The cost function is $C^2$. %\sout{ The cost function is $\mathcal{C}^2(X\times Y)$ and twisted.} {\color{orange} LN: we can now use the hyp H3 (up to ask the cost to be of class $\mathcal C^2$ from the very beginning) }
\end{itemize}

\begin{proposition}
\label{prop:derivative}
 Let $c,\rho,\nu$ verify (H1)-(H5), $Y$ be generic with respect to both $c$ and $\partial X$, and $c$ be twisted. Then  $\frac{\partial}{\partial t}\nabla_\psi \Phi(\psi (t),t) \rightarrow 0$ as $t\rightarrow 1$ where $\psi(t)$ is the solution to \eqref{cauchy}. 
\end{proposition}
 The proof of the proposition is very similar to the one for the quadratic cost of \cite{delalande2022nearly}[Proposition 5.1]; the main difference consists in partitioning the space $X$, to estimate the integral in \eqref{eq:der}, by looking at what happens on the tangent space $x\mapsto \nabla_x c(x,y_j)$ for every $x\in X$ and for each component $j$ of the derivative. For the sake of completeness  we have detailed the proof in Appendix \ref{appendix}.

\begin{proposition}\label{prop: ODE up to 1}
If $c$ is twisted, $Y$ is generic with respect to both $c$ and $\partial X$ and assumptions (H1)-(H4) are satisfied, the ODE in \eqref{cauchy} is satisfied on
$[0,1]$.
\end{proposition}
\begin{proof}
We need only to show that the ODE is satisfied at $t=1$.  It is well known that $\psi(t) \rightarrow \psi(1)$, the solution to the unregularized problem (e.g., see Theorem 1.1 in \cite{nutz2022entropic} and Theorem 1.3 in \cite{altschuler2022asymptotics}).  Furthermore, the conditions ensure  existence and continuity %$\mathcal C^2$ smoothness 
of the Hessian $\nabla^2_{\psi\psi}\Phi(\psi,1)$ of the unregularized problem (Theorem 47 in \cite{merigot:hal-02494446}), as well as the convergence of $\nabla^2_{\psi\psi}\Phi(\psi,t)$ to it (Theorem 54 in \cite{merigot:hal-02494446}).  Since the mixed derivative $\frac{\partial}{\partial t}\nabla_\psi \Phi(\psi(t),t) \rightarrow 0$,   the ODE we want to solve it $t=1$ is
\begin{equation}\label{eqn: ODE at 1}
\nabla^2_{\psi\psi}\Phi(\psi,1) \psi'(1) =-\frac{\partial}{\partial t}\nabla_\psi \Phi(\psi(1),1) =0
\end{equation}
Taking the limit of the ODE in \eqref{cauchy} as $t\rightarrow 1$ and using Proposition \ref{prop:derivative}, we get that $\nabla_{\psi,\psi}^2\Phi(\mathbf{\psi}(t),t)\mathbf{\psi}'(t) \rightarrow 0$, which, since $\nabla_{\psi,\psi}^2\Phi(\mathbf{\psi}(t),t)$ is bounded below on $U$, implies that $\psi'(t) \rightarrow 0$ as $t \rightarrow 1$.  We conclude that $\psi'(1) =0$.

%we can take the limit of the ODE in \eqref{cauchy} as $t\rightarrow 1$ to obtain that the ODE is satisfied at $t=1$.
\end{proof}

\begin{rem}[Rate of convergence]
  Notice that the results above  imply that one can obtain  the rate of convergence of the dual potentials, the primal solution and the entropic cost as in \cite{delalande2022nearly,altschuler2022asymptotics}, for a generic cost function.  In particular, since $\psi'(1)=0$ by Propositions \ref{prop:derivative} and \ref{prop: ODE up to 1}, we have $\|\psi(t) -\psi(1)\| =o(1-t)$.
 % {\color{red}BP: Is something off here?  Do we mean $\|\psi(t) -\psi(1)\| =o(1-t)$?L: Yes I think you are right}
\end{rem}

\begin{rem}[Mass of the Laguerre cells]
In \cite{kitagawa} the authors proved strong convexity of the Hessian of the unregularized Kantorovich functional on the set $\mathcal S_c:=\{\psi\;|\;\rho(\text{Lag}_i(\psi))\geq c>0,\,\forall i=1,...,N\}$. It would have been natural then to prove the well-posedness of the ODE on the same set for some constant $c$. In order to do this one must ensure  that for every $t$ the vector $\psi(t)$, the solution to entropic problem, belongs actually to $\mathcal S_c$. However, this result does not hold in general. Consider a problem with ten target points in 2-d and a nonuniform target density $\mu\in\mathbb{R}_{>0}^{10}$. As illustrated in Figure \ref{fig:2d_10pnts}, the plot of the masses of the Laguerre cells ($y$-axis) as functions of \(t\) ($x$-axis) along the curve \(t\mapsto\psi(t)\)  shows that some masses may vanish for some \(t\in[0,1)\), even though none of the masses are zero at \(t=1\).
 
\end{rem}

\begin{figure}[ht]
    \centering
    \includegraphics[width=0.7\linewidth]{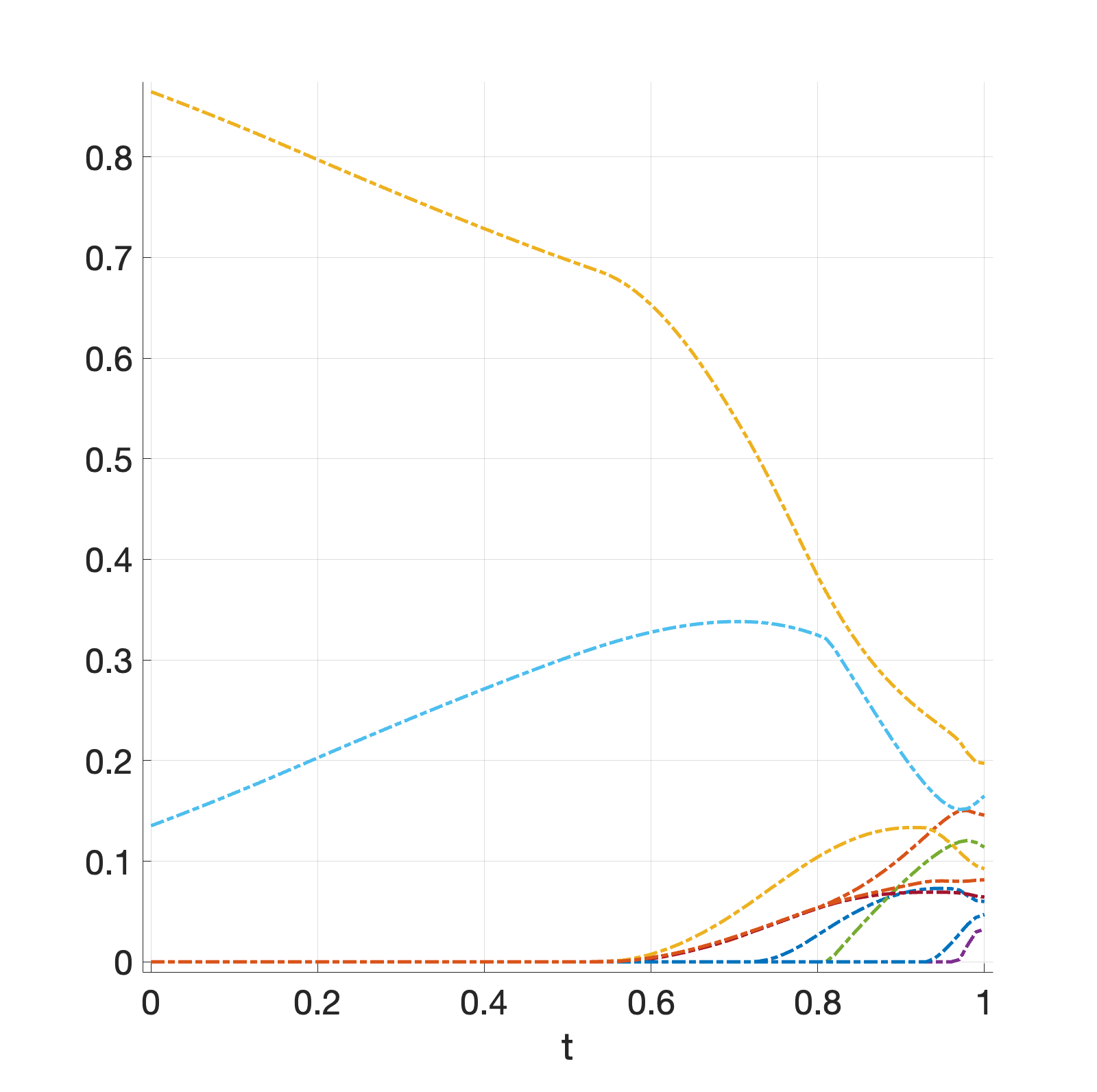}
    \caption{Mass of 10 Laguerre cells as a function of $t$ with $\dd \rho(x) = \dd x$ and $c(x,y) = ||x-y||_2^2$}
    \label{fig:2d_10pnts}
\end{figure}

\section{Computational Examples}\label{sec:compEx}
%{\color{red}BP: Why do we change notation from $\psi$ to $\mathbf{v}$ here?} {\color{orange} LN: right I think we should keep the notation $\psi$} {\color{blue} DO: done}
The initial value problem (IVP) to solve is:
\begin{gather}\label{ODE1}
    \nabla^2_{\psi\psi}\Phi(\psi(t),t)\psi'(t) + \frac{\partial}{\partial t}\nabla_\psi\Phi(\psi(t),t) = 0, \  t\in [0,1]\ ,
\end{gather}
where $\sum_{k=1}^N\psi_k(0) = 0$. Observe that the value of the initial condition can be directly determined from \eqref{eqn: gradient} by setting $t=0$ and applying the constraint $\sum_{k=1}^N\psi_k(0) = 0$. Subsequently, the initial value problem (IVP) \eqref{ODE1} can be reformulated as follows:
\begin{gather}\label{IVP}
    \psi'(t) = - [\nabla^{2,\perp}_{\psi\psi}\Phi(\psi(t),t)]^{-1}\frac{\partial}{\partial t}\nabla_\psi\Phi(\psi(t),t), \ \psi(0) = \log{\mu} - \frac{1}{N}\sum_{k=1}^N\log{\mu}, \ t\in [0,1]\ .
\end{gather}
For solving \eqref{IVP} in all examples presented below, we will use the Runge-Kutta method from Remark \ref{rem:RK} with parameters $\alpha = \frac{1}{8}$ and $\beta = \frac{1}{4}$. Furthermore, we will provide the value of the error in $\psi$: $\text{Error} = ||\psi(1) - \psi_{exact}||_{\infty}$, where $\psi_{exact}$ is the solution of the unregularized semi-discrete optimal transport problem, either determined exactly in one-dimensional examples or known a-priori for two-dimensional cases.

\begin{rem}[Third-Order Runge-Kutta Method]\label{rem:RK}
   We will use the following family of the third-order Runge-Kutta methods with $\alpha$ and $\beta$ parameters:
    \begin{gather*}
        k_1 = f(t_n, y_n)\ ,\ k_2 = f(t_n + c_2h, y_n + a_{21}hk_1)\ , \
        k_3 = f(t_n + c_3h, y_n + h(a_{31}k_1 + a_{32}k_2)\\
        \implies y_{n+1}=y_{n}+h{\bigl [}b_1k_1 + b_2k_2 + b_3k_3{\bigr ]}\ ,
    \end{gather*}
    where $\alpha, \beta\neq 0$, $\alpha\neq \beta$, $\alpha \neq \frac23$, and
    \begin{gather*}
        a_{21} = \alpha\ , \ a_{31} = \frac{\beta}{\alpha}\frac{\beta - 3\alpha(1-\alpha)}{3\alpha-2}\ , \ a_{32} = -\frac{\beta}{\alpha}\frac{\beta - \alpha}{3\alpha-2}\ ,\\
        b_1 = 1 - \frac{3\alpha + 3\beta - 2}{6\alpha\beta}\ , \ b_2 = \frac{3\beta-2}{6\alpha(\beta-\alpha)}\ , \ b_3 = \frac{2-3\alpha}{6\beta(\beta - \alpha)}\ ,\ c_1 = \alpha\ , \ c_2 = \beta\ .
    \end{gather*}
\end{rem}

\renewcommand\theequation{E\arabic{equation}} \setcounter{equation}{0}
\subsection{Problems in 1-d}
Consider the following problems on $X = [0, 1]$:
\begin{gather}\label{E1}
   \dd\rho(x) = \dd x\ ,\  y = \{0.25,\ 0.5,\ 0.75\}\ ,\  \mu = \icol{0.3\\ 0.4\\ 0.3} \ .
\end{gather}
\begin{gather}\label{E2}
    \dd\rho(x) = 1.8305e^{-10(x-0.5)^2}\dd x\ ,\  y = \{0.25,\ 0.5,\ 0.75\}\ ,\  \mu = \icol{0.3\\ 0.4\\ 0.3}\ .
\end{gather}
\begin{gather}\label{E3}
    \dd\rho(x) = \dd x\ ,\  y =\{-3.4584,\ -2.3668,\ 0.3374,\ 2.4005\}\ ,\ \mu =\icol{0.0078\\ 0.4920\\      0.4823\\ 0.0179}\ .
\end{gather}
In addition to solving the  IVP \eqref{IVP}, we will also solve the above problems using Newton's method for comparison reasons. The details of Newton's method are outlined in Remark \ref{rem:Newton1d}.

\begin{rem}[Newton's Method in 1-d]\label{rem:Newton1d}
    To solve the semi-discrete OT problem in one dimension with $c(x,y) = ||x-y||_2^p$, $p\geq 2$, one can perform the following Newton's iterations until convergence:
    \begin{gather*}
        \psi^{(k+1)} = \psi^{(k)} - [\nabla G(\psi^{(k)})]^{\dagger} G(\psi^{(k)})\ ,
    \end{gather*}
    where $[\nabla G]^{\dagger}$ represents the inverse of $\nabla G$ taken in the orthogonal complement of $\text{ker}(\nabla G) = \mathbf{1}$,
    \small
    \begin{gather*}
        \{G(\psi)\}_i = \mu_i - \rho(\text{Lag}_i(\psi)) = \mu_i - \int_{\text{Lag}_i(\psi)}\rho(x)dx\ ,\  i=1,2,\dots, N\ ,\\
        \{\nabla G(\psi)\}_{ij} = \int_{\text{Lag}_i(\psi)\cap \text{Lag}_j(\psi)}\frac{\rho(x)}{|\nabla_x c(x,y_i) - \nabla_x c(x,y_j)|}ds = \frac{\rho(x_{ij})}{|\nabla_x c(x_{ij},y_i) - \nabla_x c(x_{ij},y_j)|}\ ,\ i\neq j\ ,\\
        \{\nabla G(\psi)\}_{ii} = -\sum_{j=1,j\neq i}^N \{\nabla G(\psi^{(k)})\}_{ij} = -\sum_{j=1,j\neq i}^N\frac{\rho(x_{ij})}{|\nabla_x c(x_{ij},y_i) - \nabla_x c(x_{ij},y_j)|}\ ,\ i=1,2,\dots, N\ ,
    \end{gather*}
    \normalsize
    and $x_{ij} = \text{Lag}_i(\psi)\cup \text{Lag}_j(\psi)$. In all one-dimensional examples presented below, the iteration of Newton's method was terminated as soon as
    \[
    \lVert \psi^{(k)} - \psi_{\mathrm{exact}} \rVert_{\infty} < 10^{-11}
    \]
    was fulfilled.
\end{rem}

As demonstrated in Table \ref{tab:IVP_1d_2}, we discern a third-order convergence rate in $t$ when $c(x,y) = ||x-y||_2^2$, as anticipated. Nevertheless, the algorithm encounters a specific mesh size where it fails, attributed to integration inaccuracies. As $t\to 1$, a high-precision integrator becomes necessary because the integrand found in the derivative term and the Hessian matrix tends to resemble a delta function. For all our experiments, we utilized the \textbf{MATLAB} integration functions: \textit{integral} for 1-dimensional, \textit{integral2} for 2-dimensional, and \textit{integral3} for 3-dimensional problems. Additionally, Figure \ref{fig:Lag_1d} illustrates the progression of Laguerre cells as we near the standard unregularized semi-discrete optimal transport problem, revealing a non-monotonic evolution. 

\begin{table}[ht]
    \centering\setstretch{1.25}
    \begin{tabular}{||c||c|c|c||}\hline\hline
        $\Delta\mathbf{ t}$ & \eqref{E1} & \eqref{E2} & \eqref{E3}\\\hline\hline
        $10^{-1}$ & $1.3891*10^{- 3}$ & $3.1750*10^{- 3}$ & $1.3942*10^{-2}$ \\\hline
        $10^{-2}$ & $3.2996*10^{- 7}$ & $9.3938*10^{- 6}$ & $5.6056*10^{-4}$ \\\hline
        $10^{-3}$ & $3.9462*10^{-10}$ & $1.1194*10^{- 9}$ & $3.1528*10^{-8}$ \\\hline
        $10^{-4}$ & $6.5607*10^{-13}$ & $1.2871*10^{-12}$ & NAN\\\hline
        $10^{-5}$ & NAN & NAN & NAN \\\hline\hline
    \end{tabular}
    \caption{Error for IVP \eqref{IVP} solution with $c(x,y) = ||x-y||_2^2$}
    \label{tab:IVP_1d_2}
\end{table}

\begin{figure}[ht]
    \centering
    \begin{subfigure}{.475\textwidth}
        \centering
        \includegraphics[width=\linewidth]{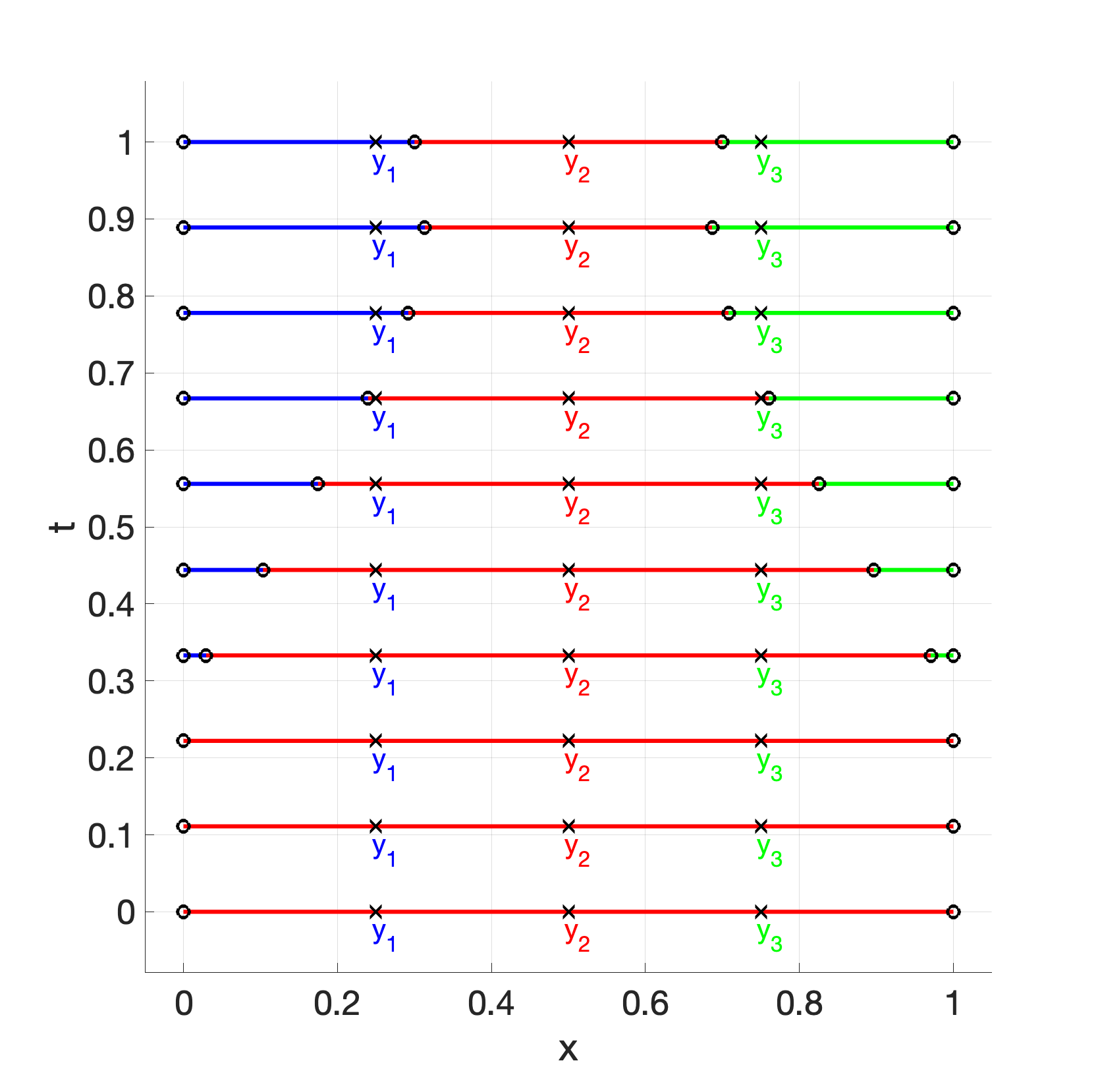}
        \caption{Example \eqref{E1}}
        \label{fig:Lag_1d_E1}
    \end{subfigure}
    \begin{subfigure}{.475\textwidth}
        \centering
        \includegraphics[width=\linewidth]{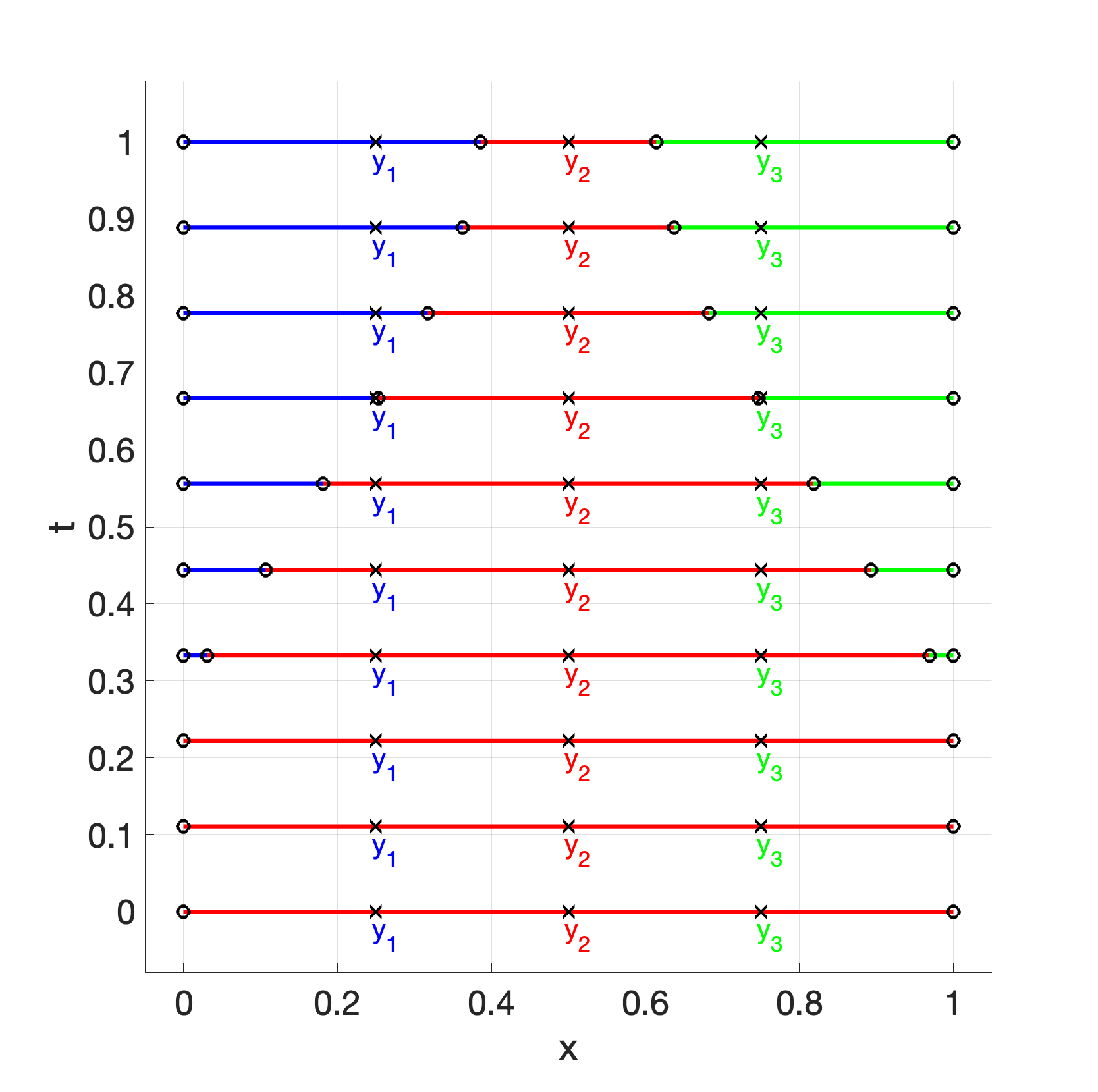}
        \caption{Example \eqref{E2}}
        \label{fig:Lag_1d_E2}
    \end{subfigure}
    \caption{Time evolution of Laguerre cells}
    \label{fig:Lag_1d}
\end{figure}

Regarding Newton's method, comparison with our approach in Table \ref{tab:Newton_1d_2} indicates similar error magnitudes in $\textbf{v}$. In addition, the computation time was of the same order as that of the initial value solver. However, Newton's method exhibits significant sensitivity to the initial conditions, leading to its failure in solving Example \eqref{E3} because the target points are outside of $X$, rendering the zero initial guess inadequate. 

\begin{table}[ht]
    \centering\setstretch{1.25}
    \begin{tabular}{||c||c|c|c||}\hline\hline
        \textbf{Initial Guess} & \eqref{E1} & \eqref{E2} & \eqref{E3}\\\hline\hline
        $10*\text{rand}(N,1)$   & $2.5647$          & $2.1454$          & NAN\\\hline
        $1*\text{rand}(N,1)$    & $3.5738*10^{-13}$ & $1.7328*10^{- 1}$ & NAN\\\hline
        $0.1*\text{rand}(N,1)$  & $5.2463*10^{-13}$ & $4.0957*10^{-13}$ & NAN\\\hline
        $0.01*\text{rand}(N,1)$ & $7.5187*10^{-13}$ & $1.8947*10^{-13}$ & NAN\\\hline
        $\mathbf{0}$            & $2.7284*10^{-13}$ & $1.8087*10^{-13}$ & NAN\\\hline\hline
    \end{tabular}
    \caption{Error for Newton's solution with $c(x,y) = ||x-y||_2^2$}
    \label{tab:Newton_1d_2}
\end{table}

Furthermore, Table \ref{tab:IVP_1d_3} presents the error convergence for $c(x,y)=||x-y||_2^3$. Here, we observe third-order convergence once again, albeit less consistently than with the squared Euclidean distance. In a parallel analysis in Table \ref{tab:Newton_1d_3}, we report the Newton's method performance under this cost. The impact of the initial guess is more pronounced; for instance, convergence to the solution in Example \eqref{E2} was only achieved when we started within $1\%$ deviation of the zero initial guess, whereas with $c(x,y)=||x-y||_2^2$, obtaining a solution was possible even with a $10\%$ perturbation.

\begin{table}[ht]
    \centering\setstretch{1.25}
    \begin{tabular}{||c||c|c|c||}\hline\hline
        $\Delta\mathbf{ t}$ & \eqref{E1} & \eqref{E2} & \eqref{E3}\\\hline\hline
        $10^{-1}$ & $7.8197*10^{- 3}$ & $7.4895*10^{- 4}$  & $7.6901*10^{-3}$ \\\hline
        $10^{-2}$ & $4.8828*10^{- 4}$ & $3.7807*10^{- 4}$  & $3.5379*10^{-6}$ \\\hline
        $10^{-3}$ & $1.7791*10^{- 8}$ & $2.3371*10^{- 7}$  & $3.5732*10^{-9}$ \\\hline
        $10^{-4}$ & $3.0193*10^{-11}$ & $4.1936*10^{-11}$  & NAN\\\hline
        $10^{-5}$ & NAN & NAN & NAN \\\hline\hline
    \end{tabular}
    \caption{Error for IVP \eqref{IVP} solution with $c(x,y) = ||x-y||_2^{3}$}
    \label{tab:IVP_1d_3}
\end{table}

\begin{table}[ht]
    \centering\setstretch{1.25}
    \begin{tabular}{||c||c|c|c||}\hline\hline
        \textbf{Initial Guess} & \eqref{E1} & \eqref{E2} & \eqref{E3}\\\hline\hline
        $10*\text{rand}(N,1)$   & $2.7753$          & $1.6201$          & NAN\\\hline
        $1*\text{rand}(N,1)$    & $1.9808*10^{- 1}$ & $4.5294*10^{- 1}$ & NAN\\\hline
        $0.1*\text{rand}(N,1)$  & $8.1341*10^{-13}$ & $2.3304*10^{- 1}$ & NAN\\\hline
        $0.01*\text{rand}(N,1)$ & $2.6465*10^{-13}$ & $5.6829*10^{-13}$ & NAN\\\hline
        $\mathbf{0}$            & $6.9658*10^{-13}$ & $2.9445*10^{-13}$ & NAN\\\hline\hline
    \end{tabular}
    \caption{Error for Newton's solution with $c(x,y) = ||x-y||_2^{3}$}
    \label{tab:Newton_1d_3}
\end{table}

\subsection{Problems in 2-d} Consider the following problems on $X = [0,1]\times [0, 1]$:
\begin{gather}\label{E4}
   \dd\rho(x) = \dd x_1 \dd x_2\ ,\  y = \big{\{}\icol{0\\0}, \icol{0\\1}, \icol{1\\1}\}\ ,\ \mu = \icol{\frac12(1-b)\\ b\\\frac12(1-b)}\ ,\ \psi_{exact} = \icol{\frac13(1-2\sqrt{b})\\-\frac23(1-2\sqrt{b})\\\frac13(1-2\sqrt{b})}\ ,
\end{gather}
where $0<b<1$,
\begin{gather}\label{E5}
    \dd\rho(x) = \dd x_1 \dd x_2\ ,\  y = \big{\{}\icol{0\\0}, \icol{0\\1}\}\ ,\ \mu = \icol{0.726759\\0.273241}\ ,\ \psi_{exact} = \icol{\frac14\ \\ -\frac14}\ ,
\end{gather}
\begin{gather}\label{E6}
    \dd\rho(x) = \dd x_1 \dd x_2\ ,\  y = \big{\{}\icol{0\\0}, \icol{0\\1}\}\ ,\ \mu = \icol{0.872066\\0.127934}\ ,\ \psi_{exact} = \icol{\frac12\ \\ -\frac12}\ .
\end{gather}

In the above examples, exact solutions are obtained for the cost function $c(x,y) = ||x-y||_2^2$ in Example \eqref{E4}, and for the cost function $c(x,y) = ||x-y||_2^4$ in Examples \eqref{E5} and \eqref{E6}. Analogously to the 1-d examples, we compare the solution of the IVP \eqref{IVP} with the results of Newton's method. When using the squared Euclidean distance as the cost, Laguerre cells can be efficiently determined using a lifting algorithm (refer to \cite{levy}). Subsequently, we calculate the Jacobian matrix using the Centered Finite Difference scheme. However, for $c(x,y) = ||x-y||_2^p$ with $p>2$, there is no efficient computational algorithm. Consequently, we approximate the Laguerre cell on a fixed mesh and again compute the Jacobian matrix using the Centered Finite Difference scheme. 

In Table \ref{tab:IVP_2d}, a third-order convergence is apparent in all examples except Example \eqref{E6}. Additionally, we note that precise integration becomes increasingly critical as the method fails even with a mesh size of $10^{-3}$, unlike the one-dimensional cases. Furthermore, Figure \ref{fig:Lag_2d_E4} illustrates the temporal progression of Laguerre partitions as $t\to 1$, exhibiting similar non-monotonic behavior as seen in the one-dimensional scenario.

\begin{table}[ht]
    \centering\setstretch{1.25}
    \begin{tabular}{||c||c|c|c|c||}\hline\hline
        $\Delta\mathbf{ t}$ & \eqref{E4} with $b=0.5$ & \eqref{E4} with $b=0.1$ & \eqref{E5}  & \eqref{E6}\\\hline\hline
        $10^{-1}$ & \makecell{$9.5043*10^{-4}$\\ $ 0.41$ sec.} & \makecell{$4.2964*10^{-4}$\\ $ 0.41$ sec.}  & \makecell{$2.8436*10^{-4}$\\ $ 0.72$ sec.} & \makecell{$3.8182*10^{-4}$\\ $ 0.70$ sec.} \\\hline
        $10^{-2}$ & \makecell{$1.9683*10^{-8}$\\ $10.13$ sec.} & \makecell{$2.1324*10^{-7}$\\ $12.46$ sec.}  & \makecell{$7.9821*10^{-8}$\\ $16.60$ sec.} & \makecell{$1.3475*10^{-5}$\\ $14.88$ sec.} \\\hline
        $10^{-3}$ & NAN & NAN  & NAN & NAN \\\hline\hline
    \end{tabular}
    \caption{Error for IVP \eqref{IVP} solution}
    \label{tab:IVP_2d}
\end{table}

\begin{figure}[ht]
    \centering
    \begin{subfigure}{.19\textwidth}
        \centering
        \includegraphics[width=\linewidth]{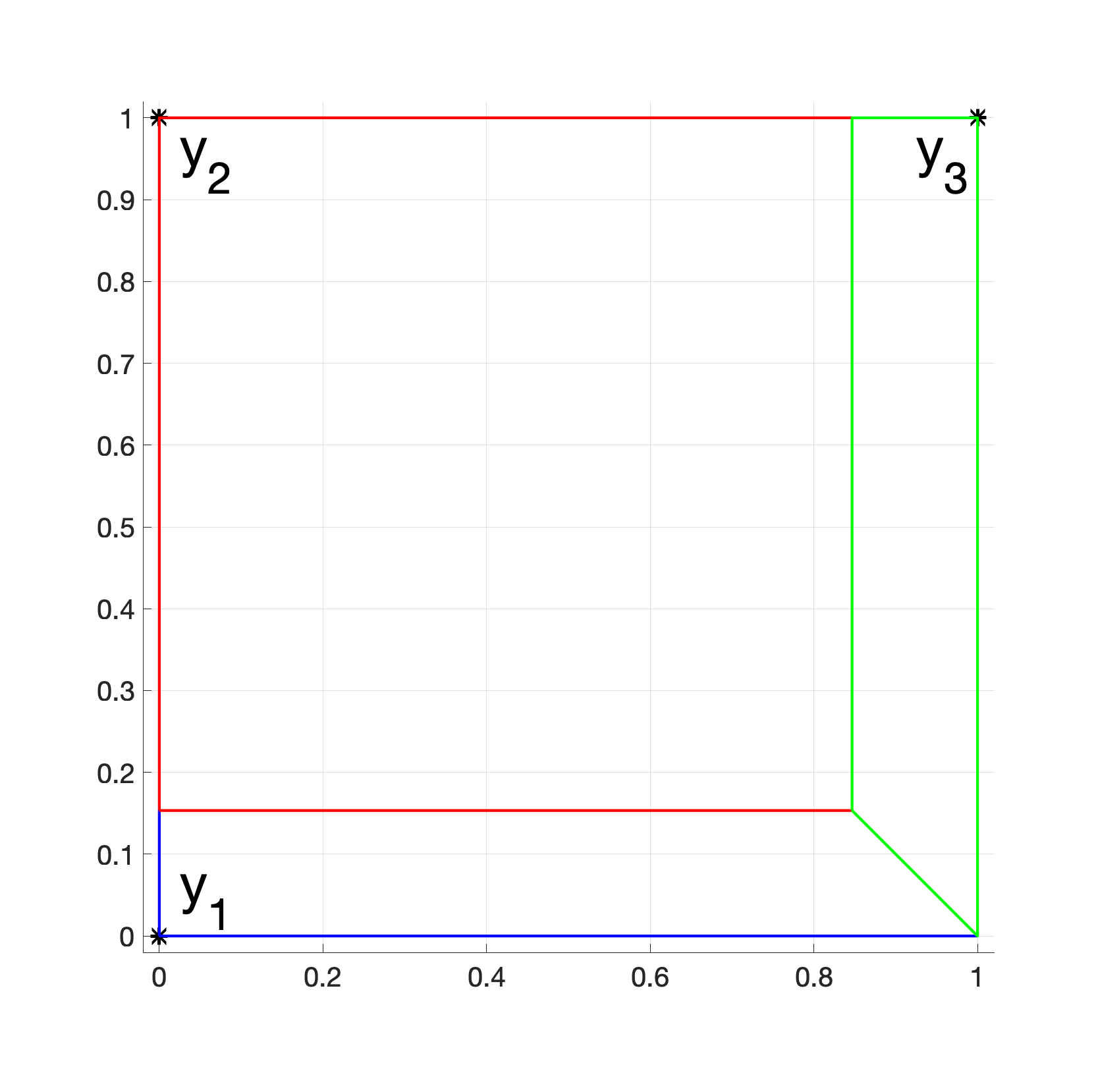}
        \caption{$t=0$}
        \label{fig:Lag_2d_E4_1}
    \end{subfigure}
    \begin{subfigure}{.19\textwidth}
        \centering
        \includegraphics[width=\linewidth]{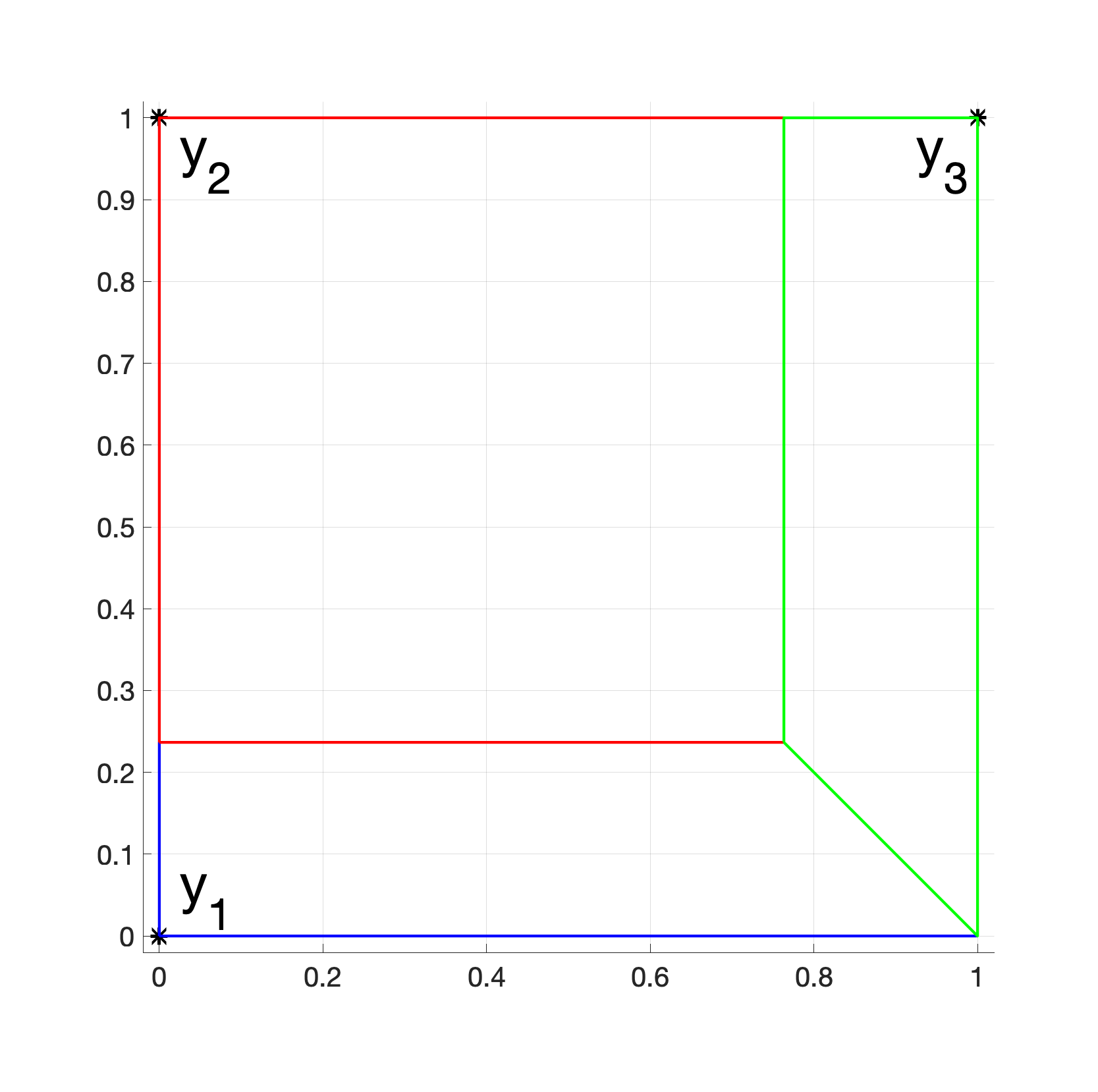}
        \caption{$t=0.25$}
        \label{fig:Lag_2d_E4_2}
    \end{subfigure}
    \begin{subfigure}{.19\textwidth}
        \centering
        \includegraphics[width=\linewidth]{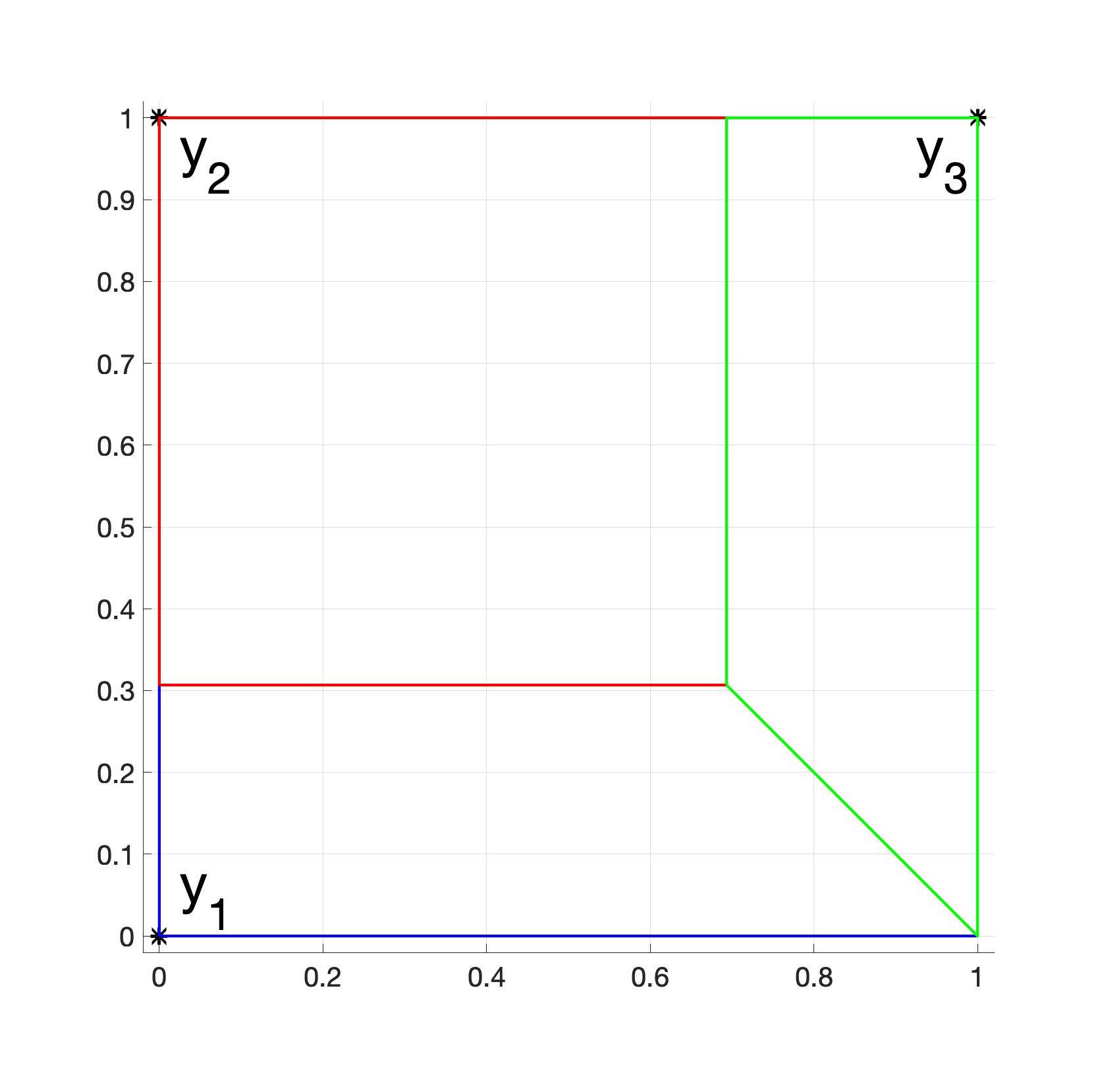}
        \caption{$t=0.5$}
        \label{fig:Lag_2d_E4_3}
    \end{subfigure}
    \begin{subfigure}{.19\textwidth}
        \centering
        \includegraphics[width=\linewidth]{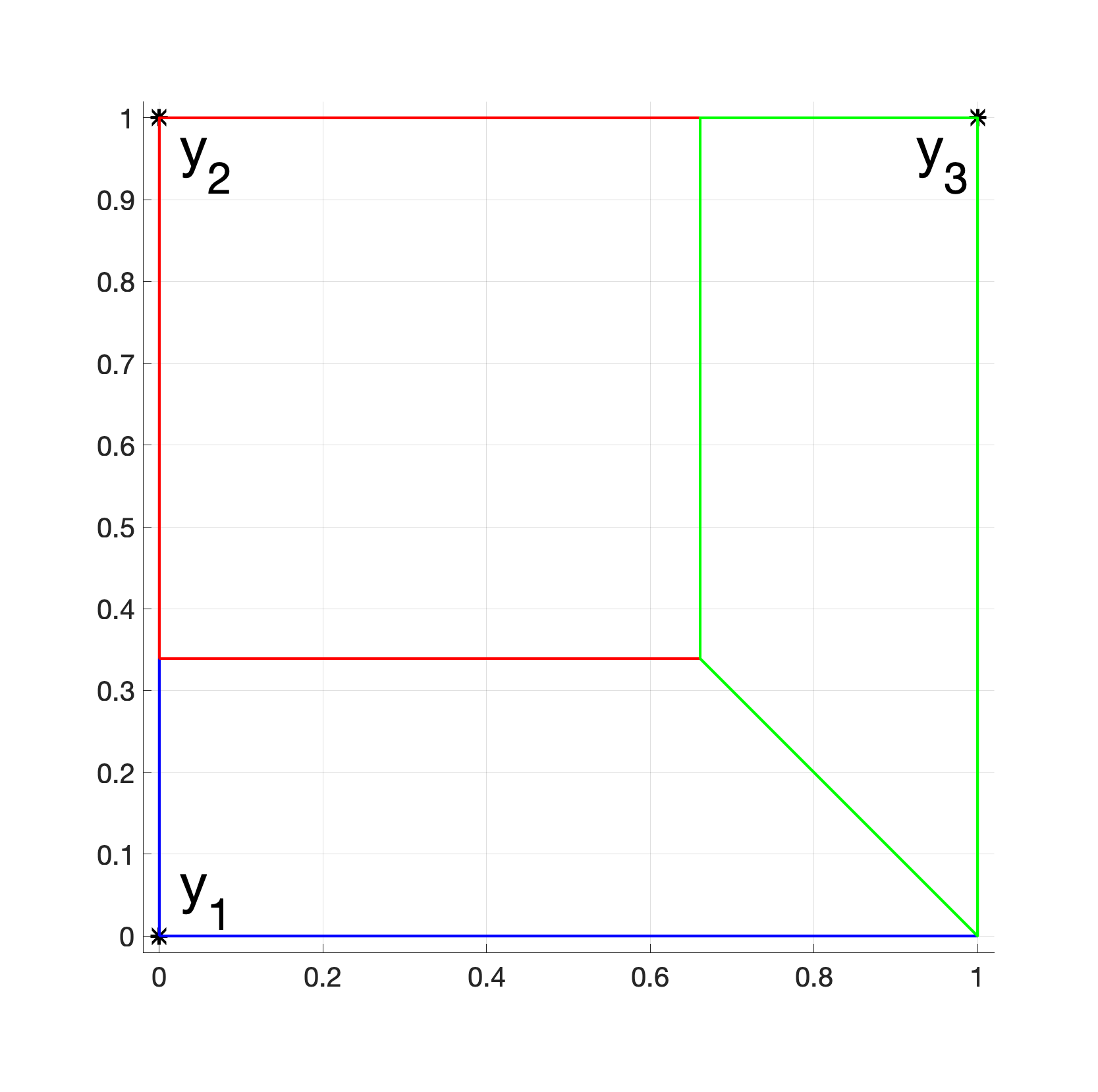}
        \caption{$t=0.75$}
        \label{fig:Lag_2d_E4_4}
    \end{subfigure}
    \begin{subfigure}{.19\textwidth}
        \centering
        \includegraphics[width=\linewidth]{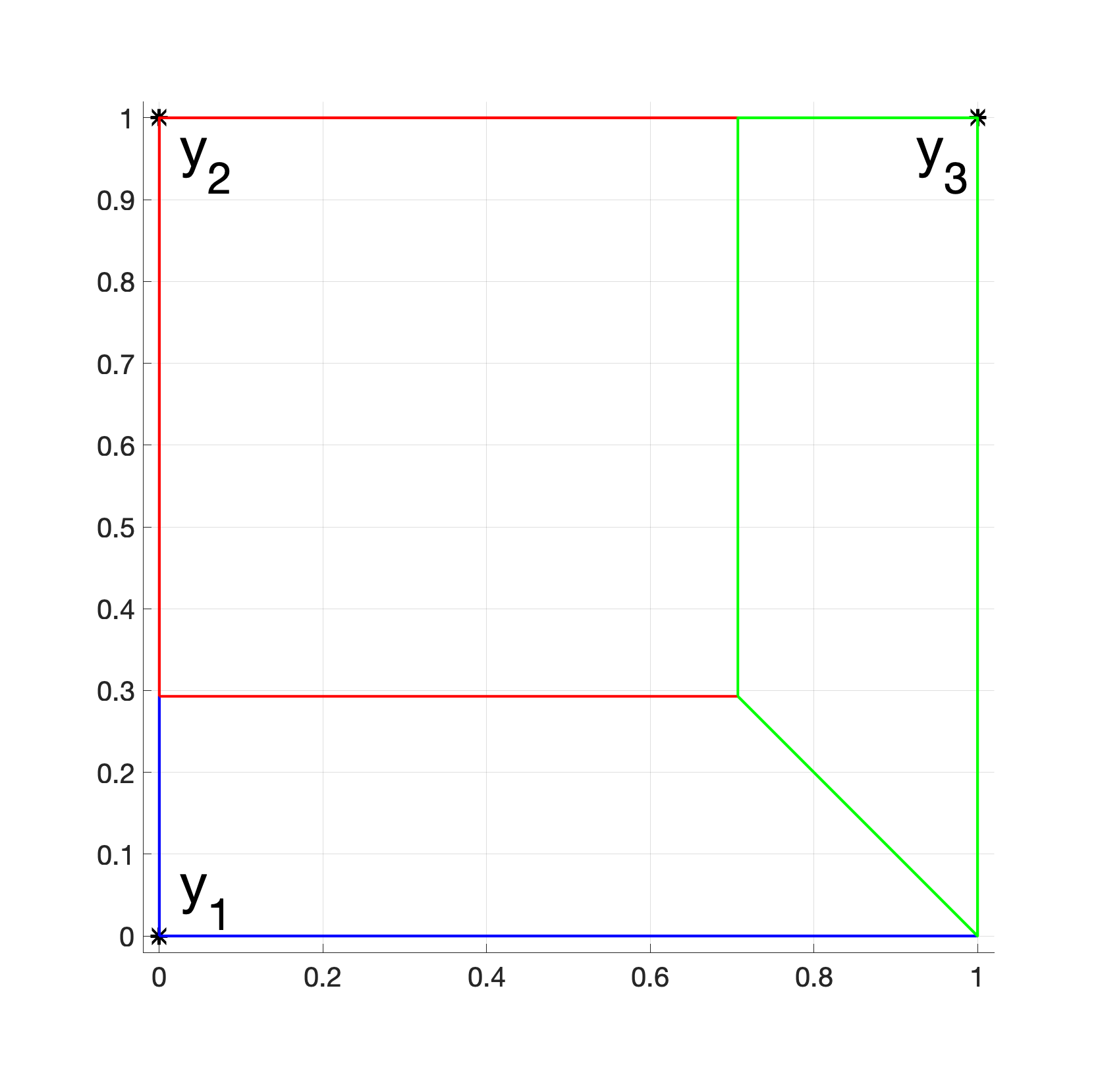}
        \caption{$t=1$}
        \label{fig:Lag_2d_E4_5}
    \end{subfigure}
    \caption{Time Evolution of Laguerre cells in Example \eqref{E4} with $b=0.5$}
    \label{fig:Lag_2d_E4}
\end{figure}

Next, we addressed the two-dimensional problems using Newton's method. In these examples, Newton's method was terminated once 
\(\lVert \psi^{(k)} - \psi_{\mathrm{exact}} \rVert_{\infty} < 10^{-7}\). As demonstrated in Table \ref{tab:Newton_2d}, the performance of Newton's method in terms of accuracy and time is comparable to our approach. Nonetheless, the importance of an initial guess becomes crucial in two-dimensional problems, as convergence to a solution in Example \eqref{E6} is not achieved. Overall, it is feasible to incorporate both methods by using the ODE solution as the initial guess for Newton's method.

Finally, to demonstrate robustness of ODE solver as the number of points increases. Figure \ref{fig:Meas_2d} illustrates the convergence of the Laguerre cell measure to  $\mu$ over $t$. In this problem, target points are randomly distributed within the unit square, both source and target measures are uniform, the cost function is the squared Euclidean distance, and the time step is $\Delta t = 10^{-2}$. The exact potential $\psi_{\text{exact}}$ is unknown; however, we can evaluate the measure error $||\rho(\text{Lag}(\psi(1))) - \mu||_{\infty}$. For instance, this error is $4.6587 \times 10^{-4}$ for 10 random target points and $1.4668 \times 10^{-3}$ for 25 target points. Additionally, Figures \ref{fig:Lag_2d_10pnts} and \ref{fig:Lag_2d_25pnts} illustrate the evolution of Laguerre cells with respect to $t$.  

\begin{table}[ht]
    \centering\setstretch{1.4}
    \begin{tabular}{||c||c|c|c|c||}\hline\hline
        \textbf{Initial Guess} & \eqref{E4} with $b=0.5$ & \eqref{E4} with $b=0.1$ & \eqref{E5}  & \eqref{E6}\\\hline\hline
        $1*\text{rand}(N,1)$    & \makecell{$1.9160*10^{-8}$\\ $7.89$ sec.} & \makecell{$2.1022*10^{-8}$\\ $7.36$ sec.} & \makecell{$1.9524*10^{-5}$\\ $15.00$ sec.} & NAN \\\hline  
        $0.1*\text{rand}(N,1)$  & \makecell{$1.3244*10^{-8}$\\ $6.97$ sec.} & \makecell{$1.5153*10^{-8}$\\ $6.66$ sec.} & NAN & NAN \\\hline
        $0.01*\text{rand}(N,1)$ & \makecell{$7.5942*10^{-9}$\\ $7.10$ sec.} & \makecell{$1.1396*10^{-8}$\\ $6.86$ sec.} & NAN & NAN \\\hline
        $\mathbf{0}$            & \makecell{$5.7658*10^{-9}$\\ $6.70$ sec.} & \makecell{$2.1945*10^{-8}$\\ $8.05$ sec.} & NAN & NAN \\\hline\hline
    \end{tabular}
    \caption{Error for Newton's solution}
    \label{tab:Newton_2d}
\end{table}

\begin{figure}[ht]
    \centering
    \begin{subfigure}{.49\textwidth}
        \centering
        \includegraphics[width=\linewidth]{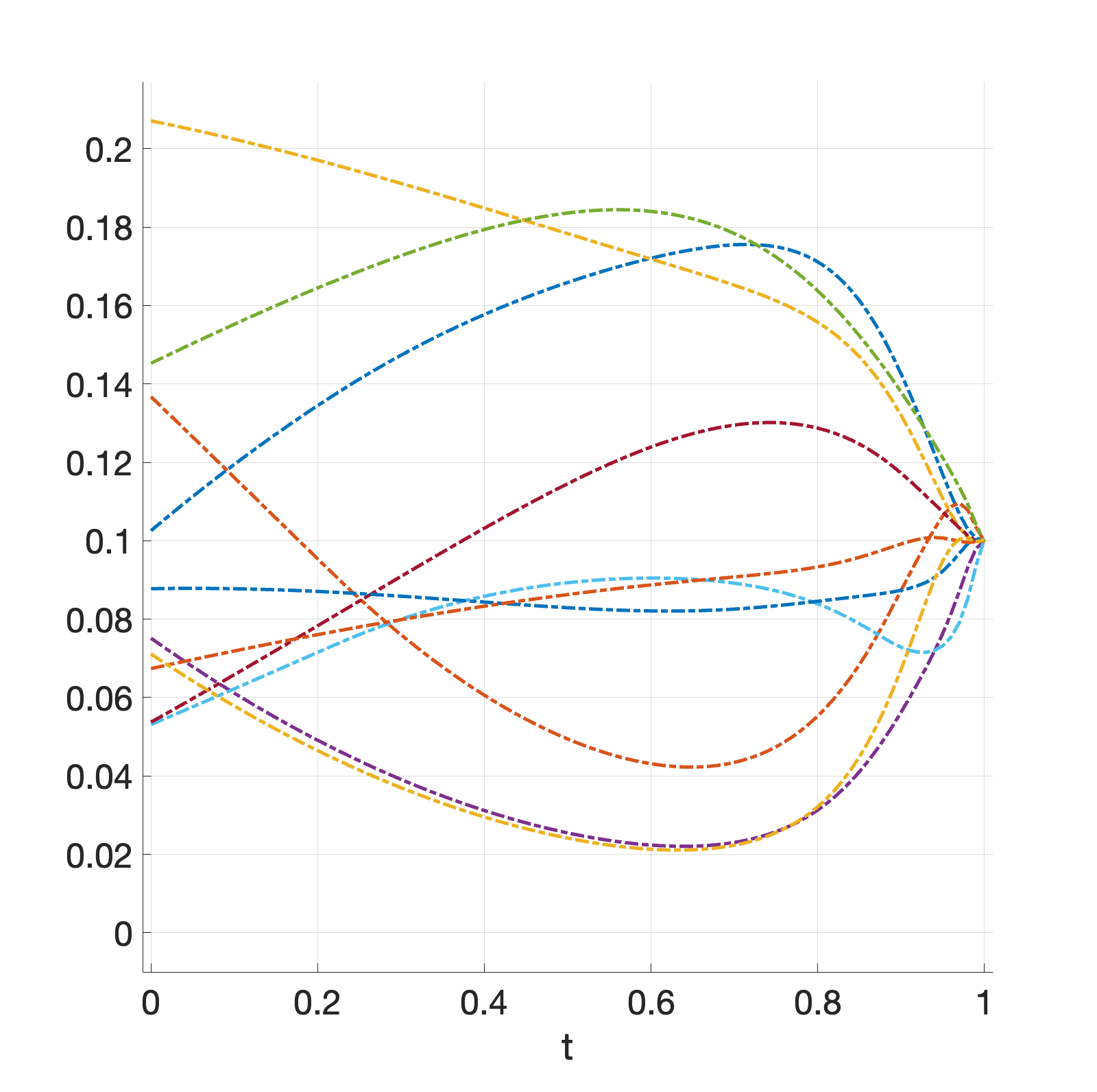}
        \caption{10 points}
        \label{fig:Meas_2d_10pnts}
    \end{subfigure}
    \begin{subfigure}{.49\textwidth}
        \centering
        \includegraphics[width=\linewidth]{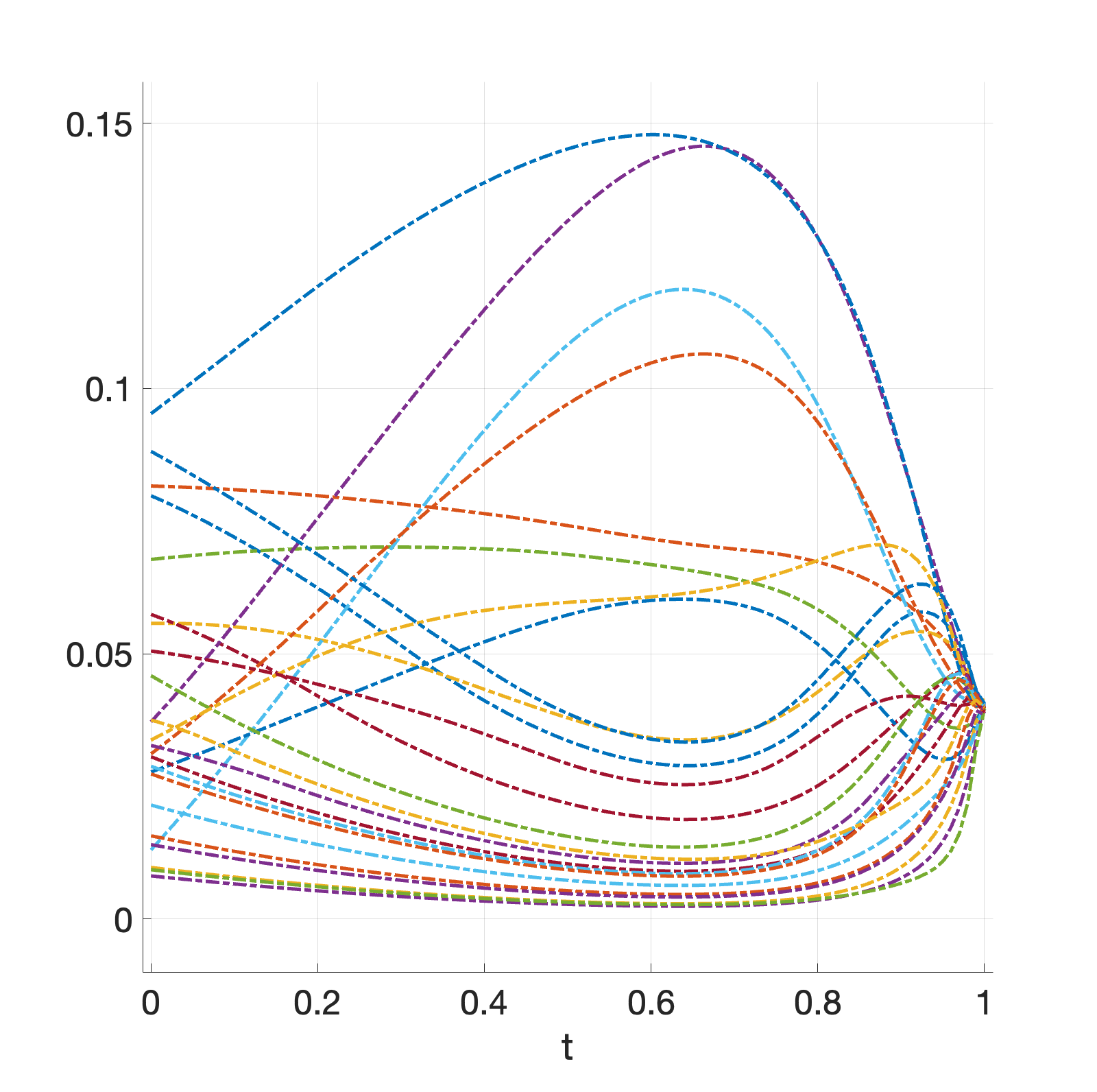}
        \caption{25 points}
        \label{fig:Meas_2d_25pnts}
    \end{subfigure}
    \caption{Time evolution of measures}
    \label{fig:Meas_2d}
\end{figure}

\begin{figure}[ht]
    \centering
    \begin{subfigure}{.19\textwidth}
        \centering
        \includegraphics[width=\linewidth]{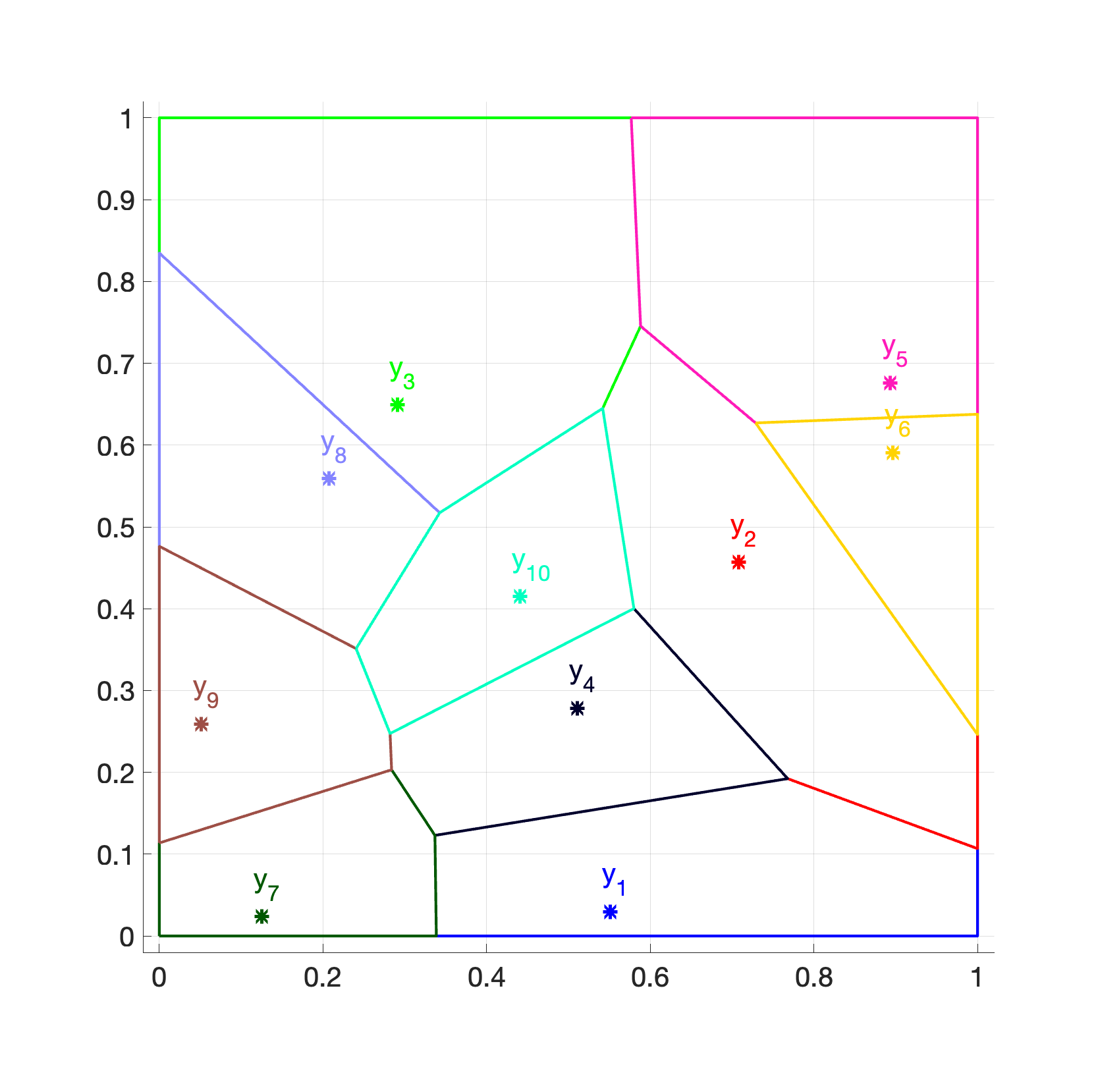}
        \caption{$t=0$}
        \label{fig:2d_10pnts_1}
    \end{subfigure}
    \begin{subfigure}{.19\textwidth}
        \centering
        \includegraphics[width=\linewidth]{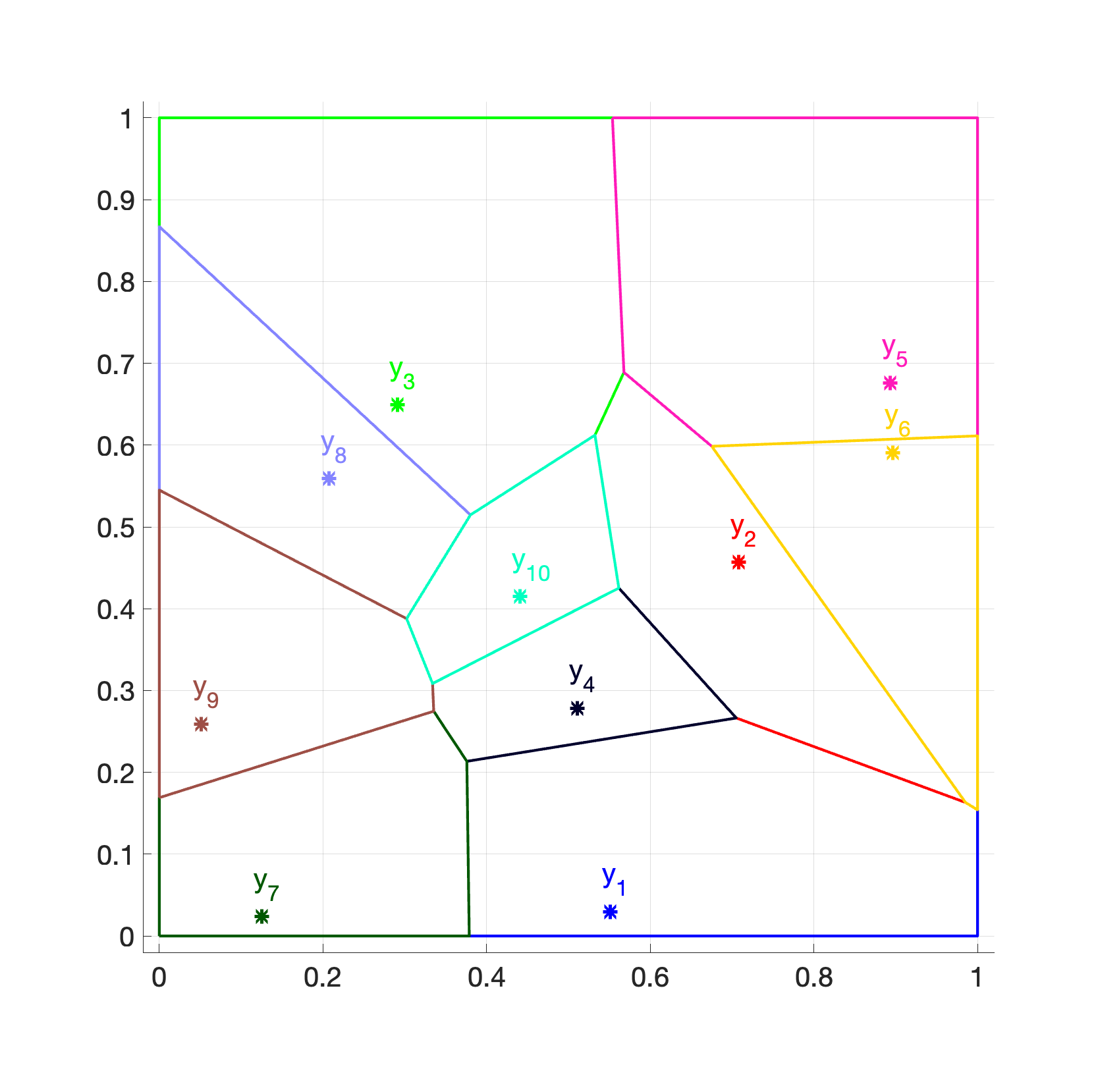}
        \caption{$t=0.25$}
        \label{fig:2d_10pnts_2}
    \end{subfigure}
    \begin{subfigure}{.19\textwidth}
        \centering
        \includegraphics[width=\linewidth]{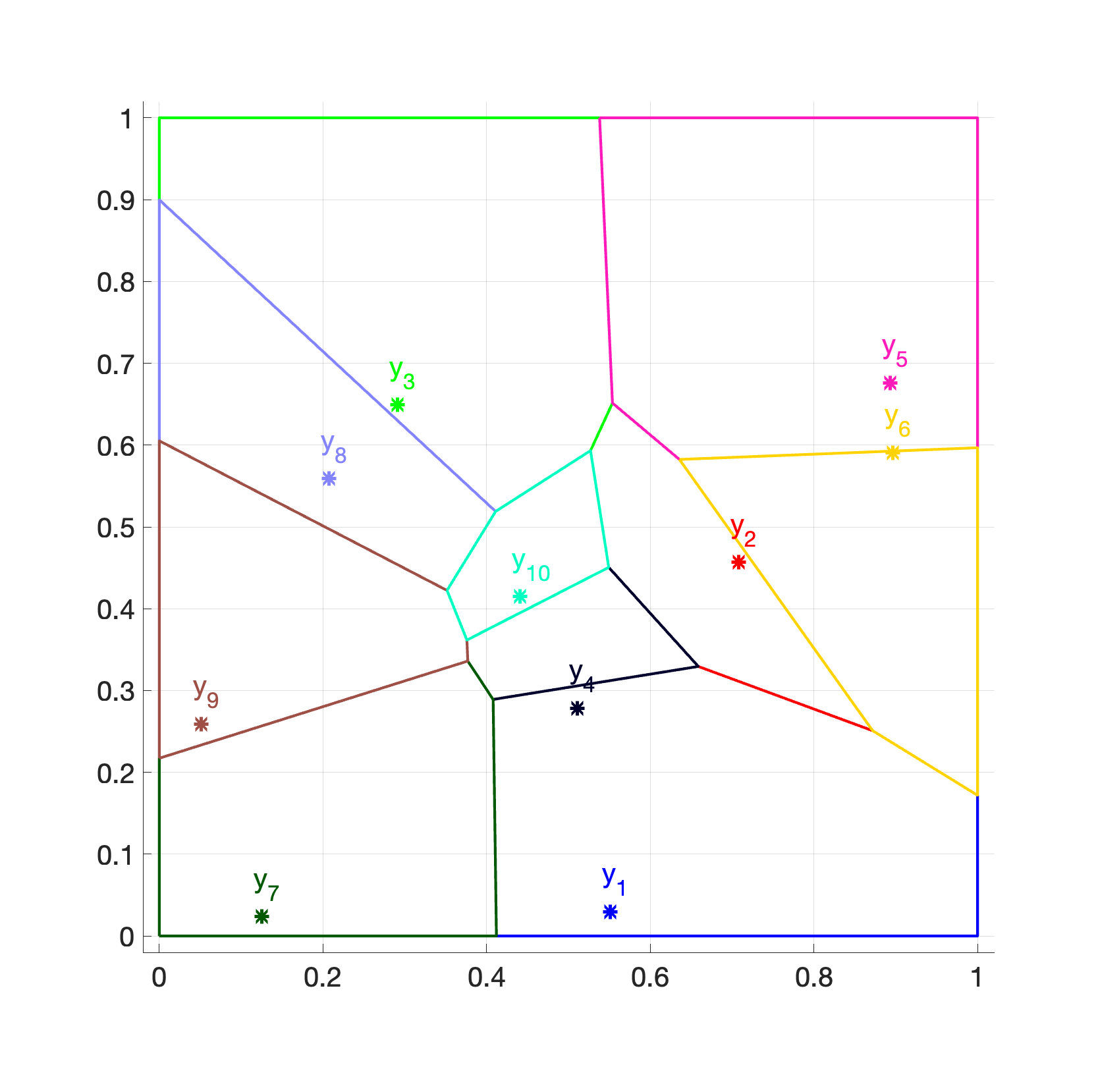}
        \caption{$t=0.5$}
        \label{fig:2d_10pnts_3}
    \end{subfigure}
    \begin{subfigure}{.19\textwidth}
        \centering
        \includegraphics[width=\linewidth]{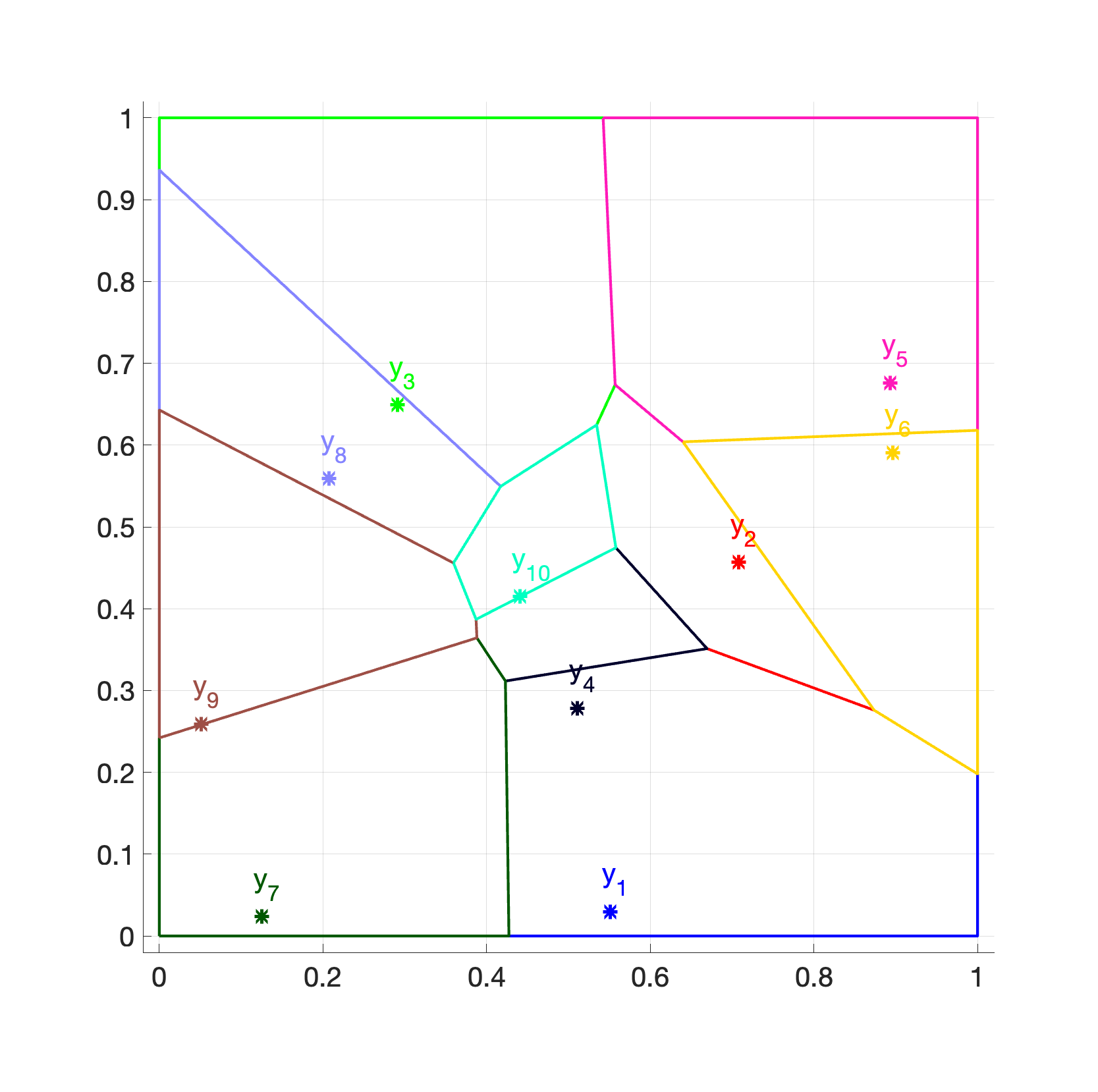}
        \caption{$t=0.75$}
        \label{fig:2d_10pnts_4}
    \end{subfigure}
    \begin{subfigure}{.19\textwidth}
        \centering
        \includegraphics[width=\linewidth]{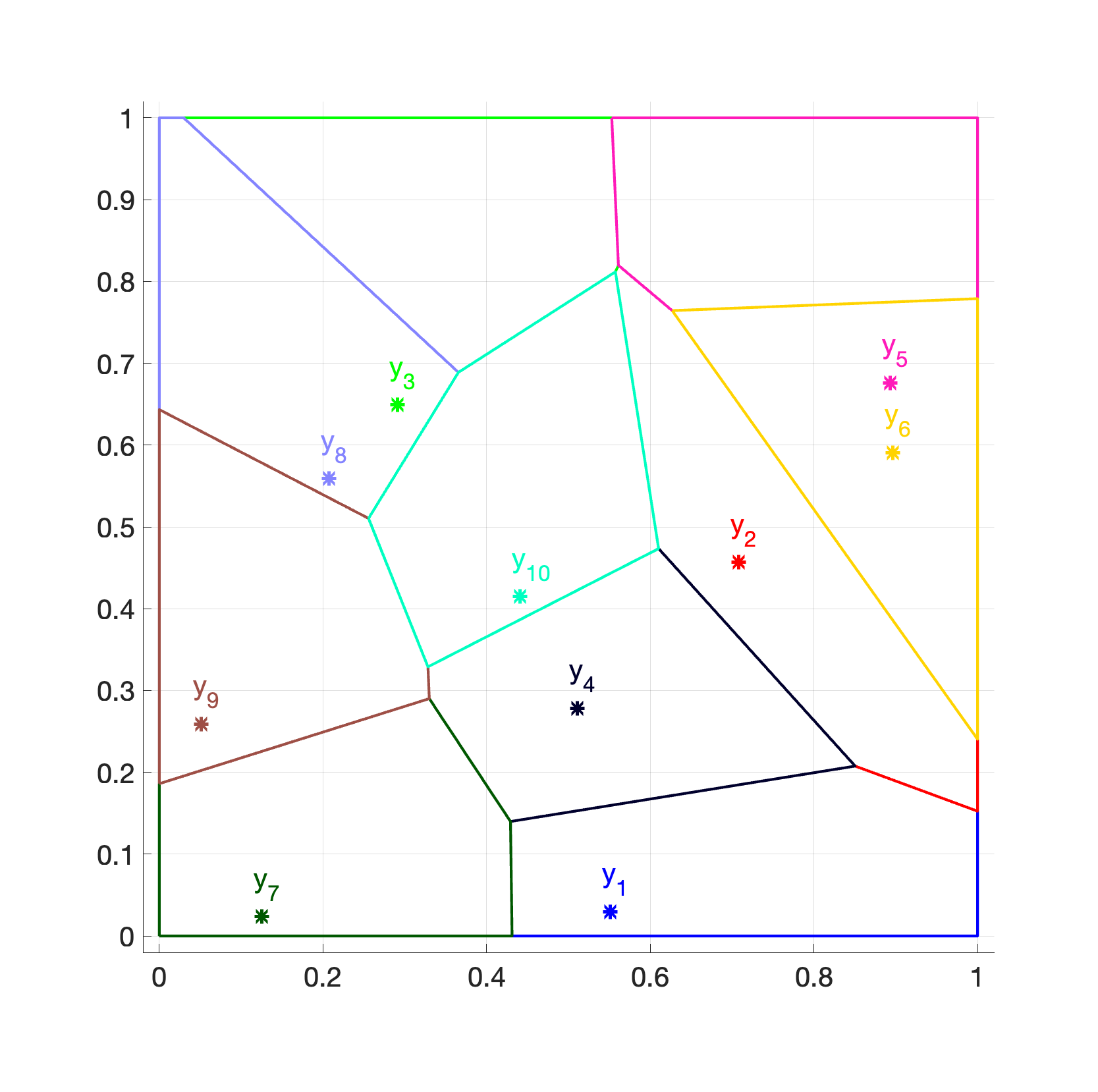}
        \caption{$t=1$}
        \label{fig:2d_10pnts_5}
    \end{subfigure}
    \caption{Time evolution of Laguerre cells with 10 random points}
    \label{fig:Lag_2d_10pnts}
\end{figure}

\begin{figure}[ht]
    \centering
    \begin{subfigure}{.19\textwidth}
        \centering
        \includegraphics[width=\linewidth]{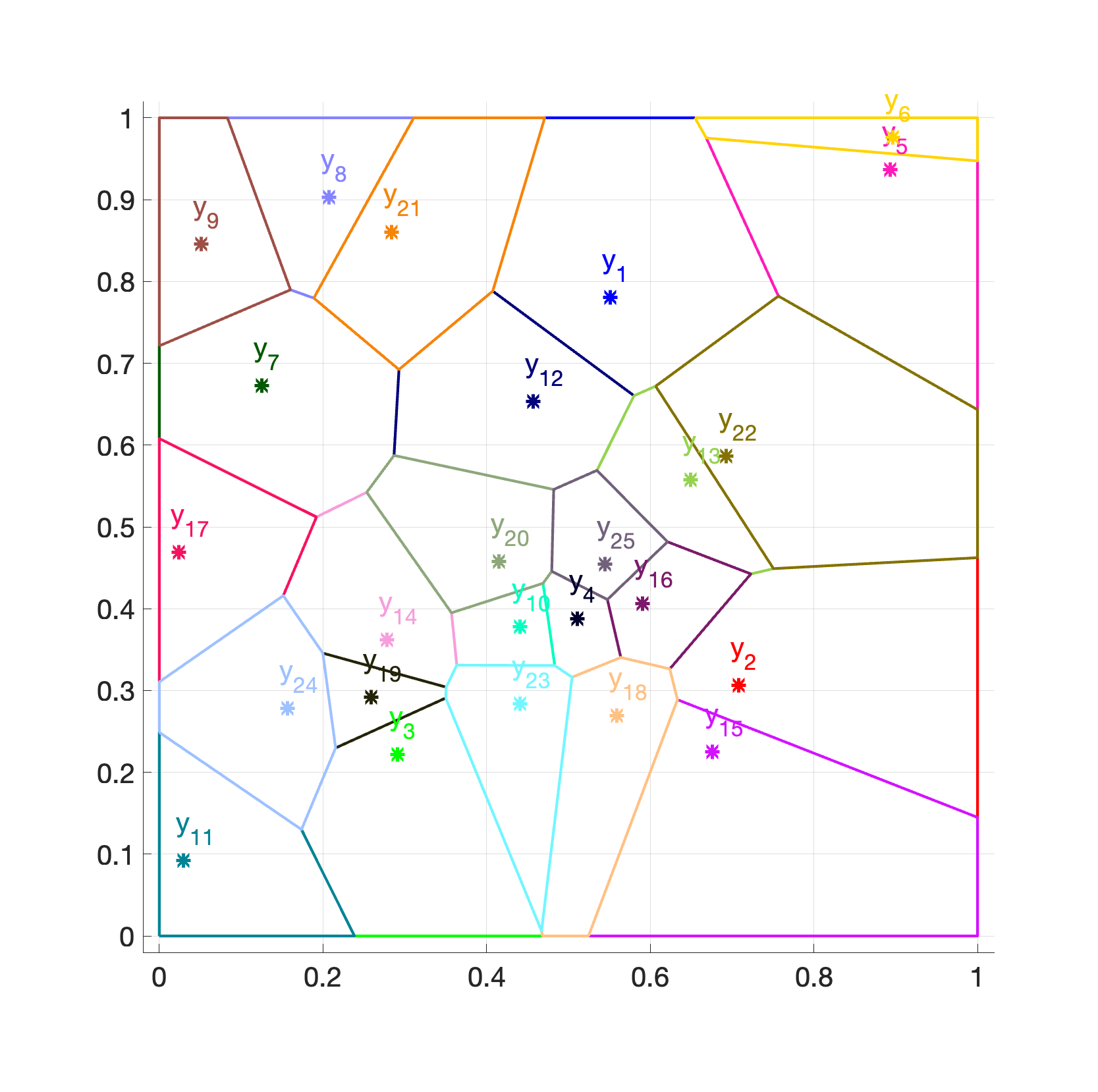}
        \caption{$t=0$}
        \label{fig:2d_25pnts_1}
    \end{subfigure}
    \begin{subfigure}{.19\textwidth}
        \centering
        \includegraphics[width=\linewidth]{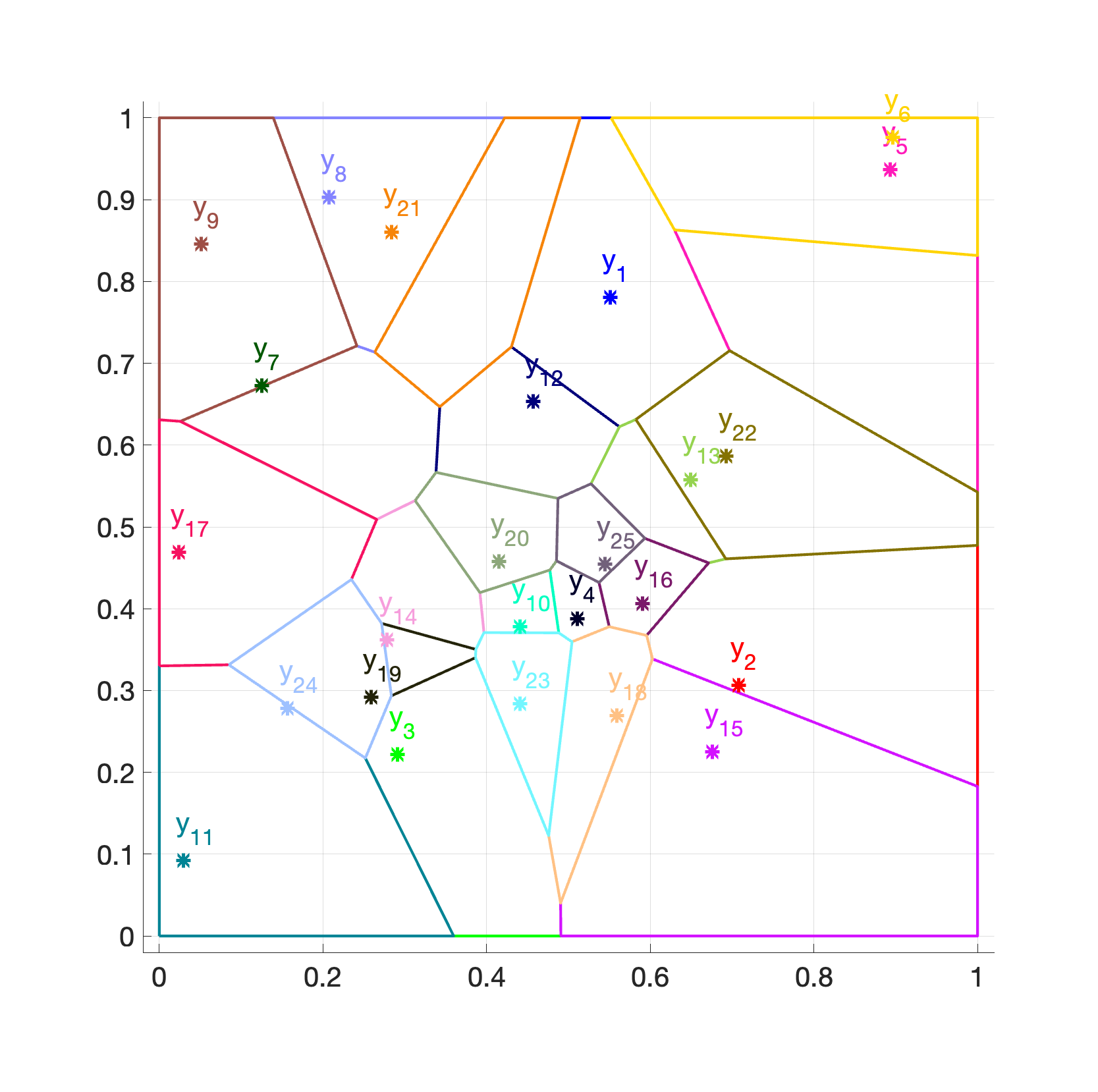}
        \caption{$t=0.25$}
        \label{fig:2d_25pnts_2}
    \end{subfigure}
    \begin{subfigure}{.19\textwidth}
        \centering
        \includegraphics[width=\linewidth]{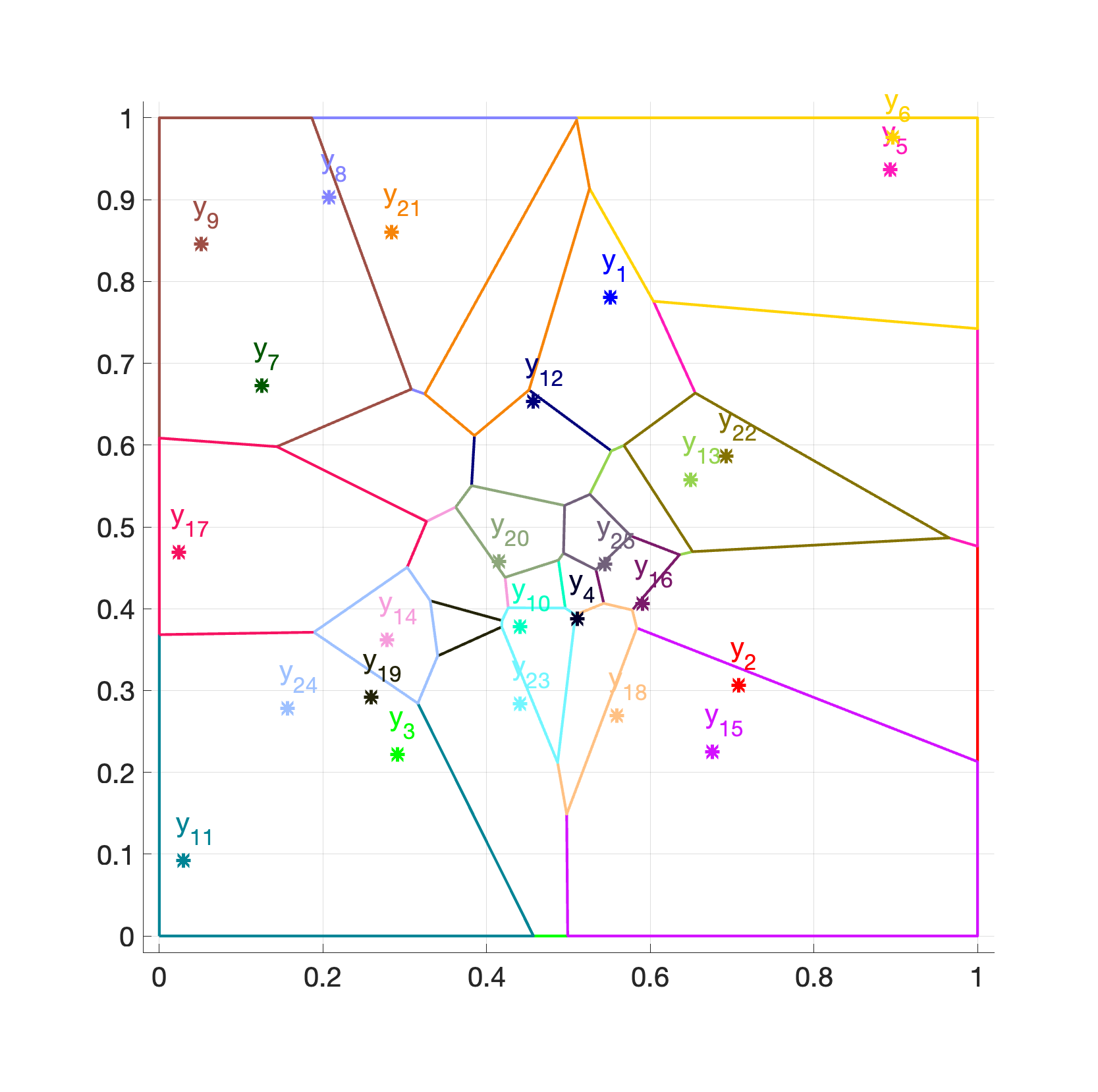}
        \caption{$t=0.5$}
        \label{fig:2d_25pnts_3}
    \end{subfigure}
    \begin{subfigure}{.19\textwidth}
        \centering
        \includegraphics[width=\linewidth]{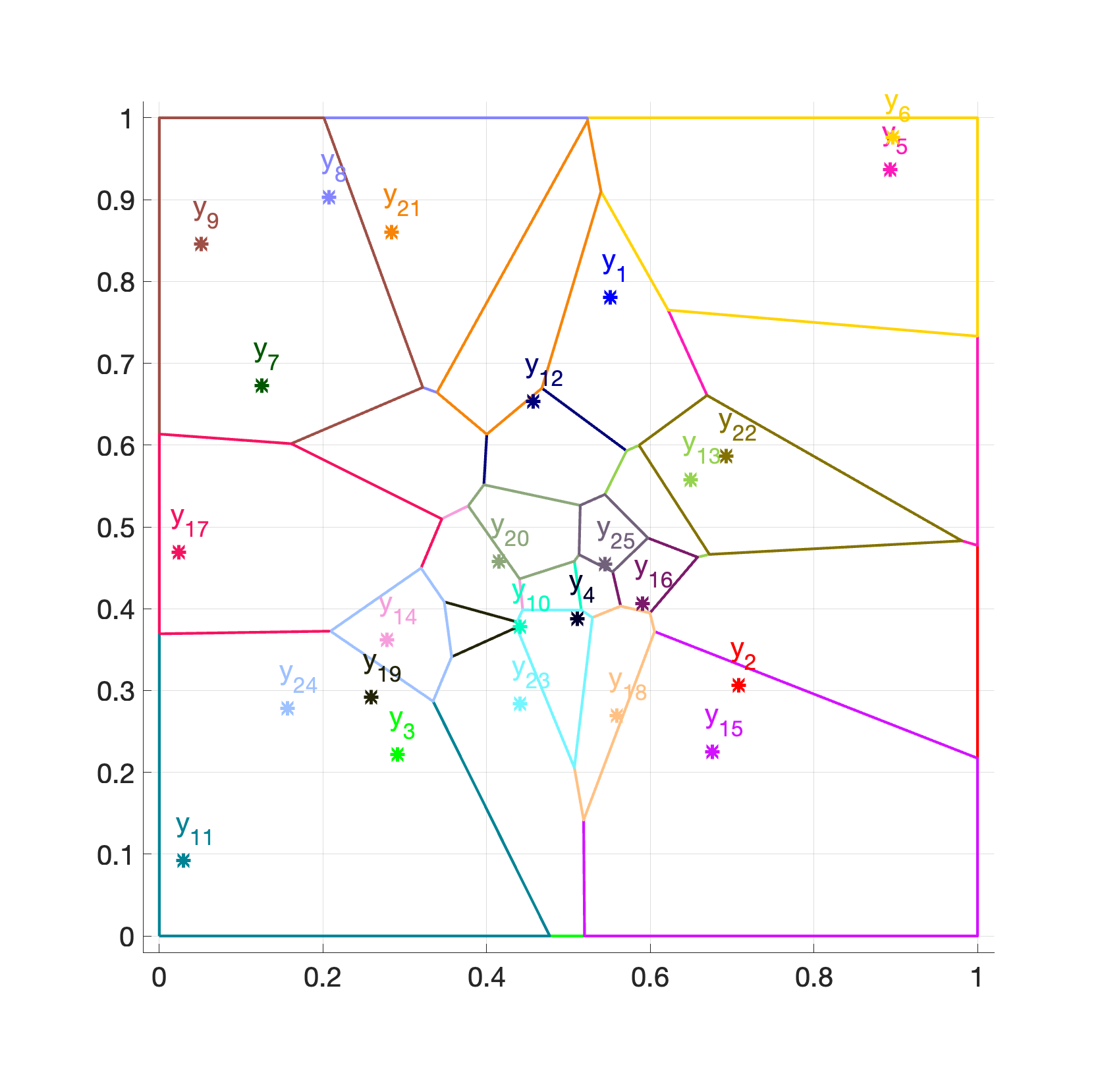}
        \caption{$t=0.75$}
        \label{fig:2d_25pnts_4}
    \end{subfigure}
    \begin{subfigure}{.19\textwidth}
        \centering
        \includegraphics[width=\linewidth]{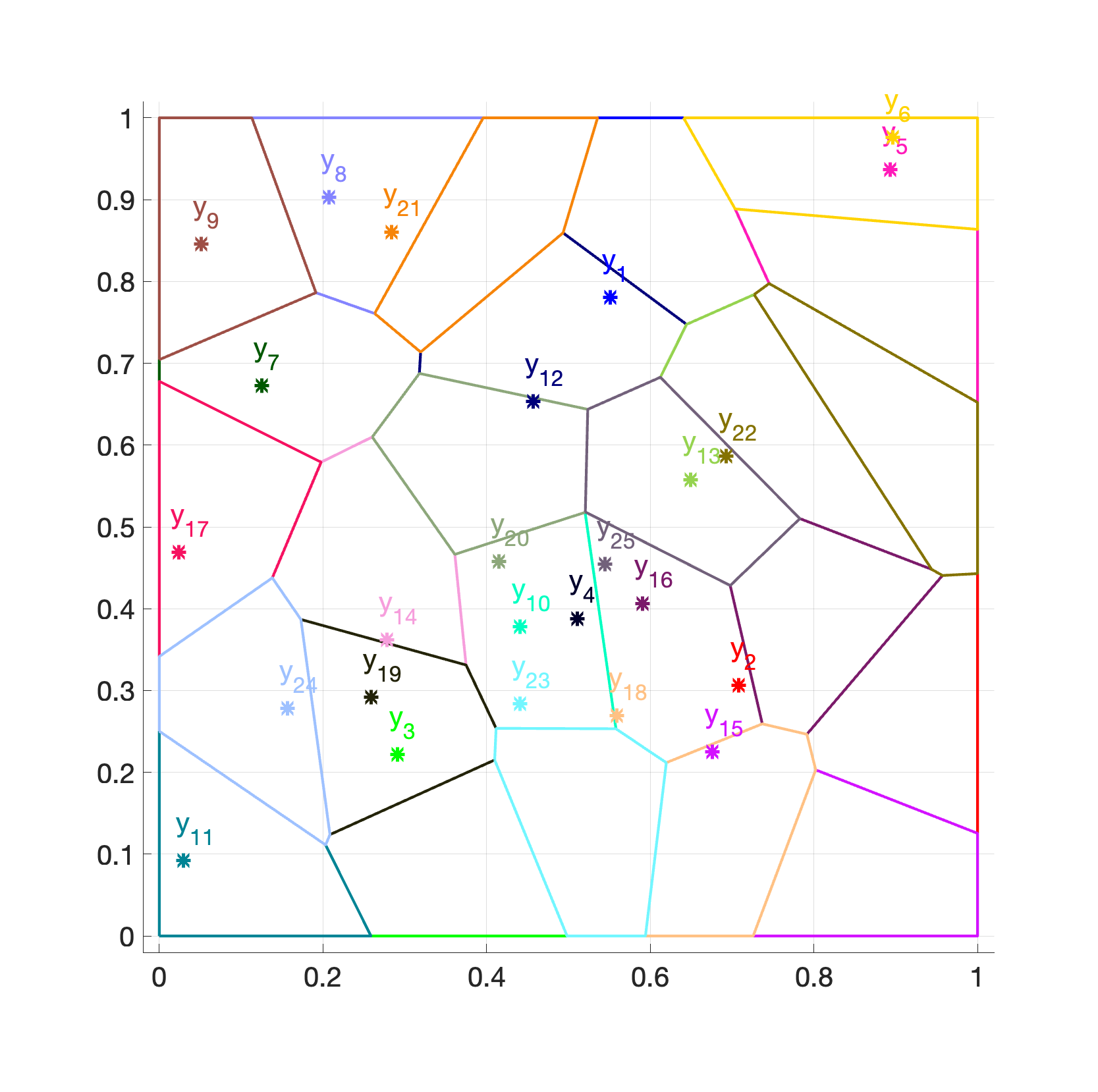}
        \caption{$t=1$}
        \label{fig:2d_25pnts_5}
    \end{subfigure}
    \caption{Time evolution of Laguerre cells with 10 random points}
    \label{fig:Lag_2d_25pnts}
\end{figure}

\subsection{Problems in 3-d}
This section delves into evaluating the effectiveness of the ODE solver in a three-dimensional context, comparing it to Newton's method, and examining the impact of dimensionality on computation time. Consider the following problem on $X = [0,1]\times[0,1]\times[0,1]$ with $c(x,y) = ||x-y||_2^2$, and
\begin{gather}\label{E7}
    \dd\rho(x) = \dd x_1\dd x_2\dd x_3\ ,\  y = \big{\{}\icol{0.5508\\0.8963\\0.0299}, \icol{0.7081\\0.1256\\0.4568}, \icol{0.2909\\0.2072\\0.6491}, \icol{0.5108\\0.0515\\0.2785}, \icol{0.8929\\0.4408\\0.6763}\}\ ,\ \mu = \frac{1}{5}\mathbf{1}\ .
\end{gather}

For non-trivial 3-d problems, the exact solution $\psi_{exact}$ is typically not available, rendering the calculation of $\text{Error} = ||\psi(1) - \psi_{exact}||_{\infty}$ infeasible. Nevertheless, it remains possible to compute the error in measure: $\text{Measure Error} = ||\rho(\text{Lag}(\psi(1))) - \mu||_{\infty}$. In this setting, the Newton iteration was terminated once the measure error satisfied 
\(< 10^{-3}\). Tables~\ref{tab:IVP_3d} and \ref{tab:Newton_3d} report the performance of the ODE method and Newton's method, respectively, indicating that the final errors are of comparable magnitude. As indicated in Table \ref{tab:IVP_3d}, there is first-order convergence towards $\mu$, suggesting second-order convergence in $\psi$ due to the integration. Moreover, the comparison of computation times reveals an order-of-magnitude difference, with Newton's method substantially outperforming the initial value solver. 
This outcome is expected, as the initial value solver additionally provides information on the evolution of the Laguerre cells from the entropy solution to the solution of the original problem (see Figure \ref{fig:Lag_3d_E7}), which is not obtained from Newton's method.

Finally, Figure \ref{fig:Time_Comparison} provides a comparative analysis of computational time across different dimensions for the ODE solution. In this scenario, the \textit{3-d Example} pertains to \eqref{E7}, while the \textit{1-d Example} and \textit{2-d Example} relate to analogous problems on the unit line and unit square, respectively: five random target points, uniform target and source measures, squared Euclidean cost, and 100 time-steps are used. It is observed that calculation time escalates exponentially with increasing dimensions. Additionally, there's a rise in computation time as $t$ nears the value of 1 in every dimension. This is attributed to the internal numerical integration process where the integrand approaches a delta function, necessitating finer spatial discretization as $t$ converges to 1.

\begin{table}[ht]
    \begin{minipage}{.5\linewidth}
        \centering\setstretch{1.25}
        \begin{tabular}{||c||c||}\hline\hline
            $\Delta\mathbf{ t}$ & Measure Error \\\hline\hline
            $10^{-1}$ & \makecell{$6.4267*10^{-3}$\\ $70.151$ sec.} \\\hline
            $10^{-2}$ & \makecell{$1.3377*10^{-3}$\\ $3348.7$ sec.} \\\hline
            $10^{-3}$ & NAN  \\\hline\hline
        \end{tabular}
        \caption{IVP \eqref{IVP} solution of Example \eqref{E7}}
        \label{tab:IVP_3d}
    \end{minipage}%
    \begin{minipage}{.5\linewidth}
        \centering\setstretch{1.25}
        \begin{tabular}{||c||c||}\hline\hline
            \textbf{Initial Guess}  & Measure Error \\\hline\hline
            $0.1*\text{rand}(N,1)$  & \makecell{$3.2875*10^{-5}$\\ $348.09$ sec.} \\\hline
            $0.01*\text{rand}(N,1)$ & \makecell{$1.6175*10^{-4}$\\ $238$ sec.} \\\hline
            $\mathbf{0}$            & \makecell{$1.7587*10^{-4}$\\ $235.97$ sec.} \\\hline
            \hline
        \end{tabular}
        \caption{Newton solution of Example \eqref{E7}}
        \label{tab:Newton_3d}
    \end{minipage} 
\end{table}

\begin{figure}[ht]
    \centering
    \begin{subfigure}{.19\textwidth}
        \centering
        \includegraphics[width=\linewidth]{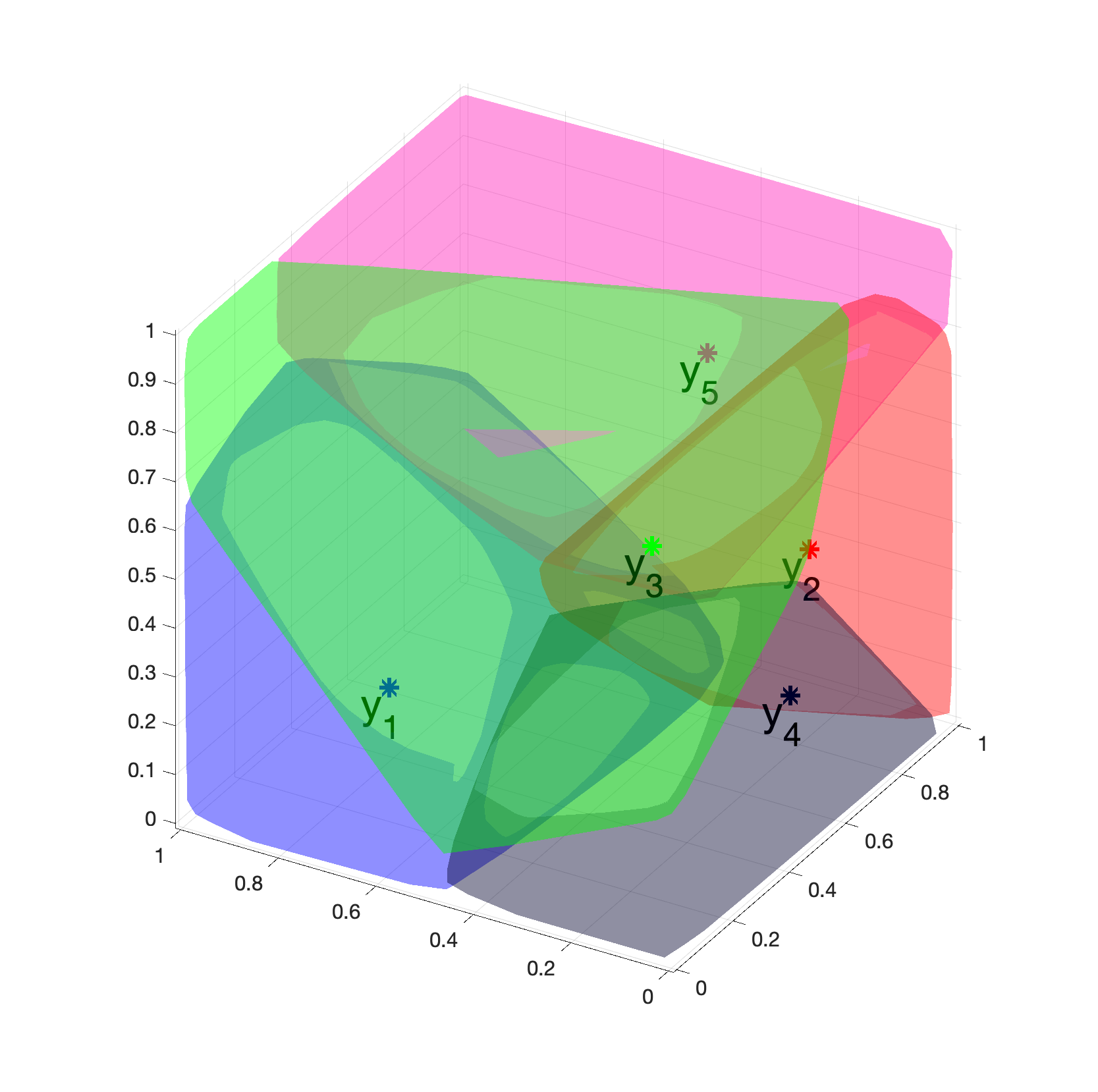}
        \caption{$t=0$}
        \label{fig:3d_5pnts_1}
    \end{subfigure}
    \begin{subfigure}{.19\textwidth}
        \centering
        \includegraphics[width=\linewidth]{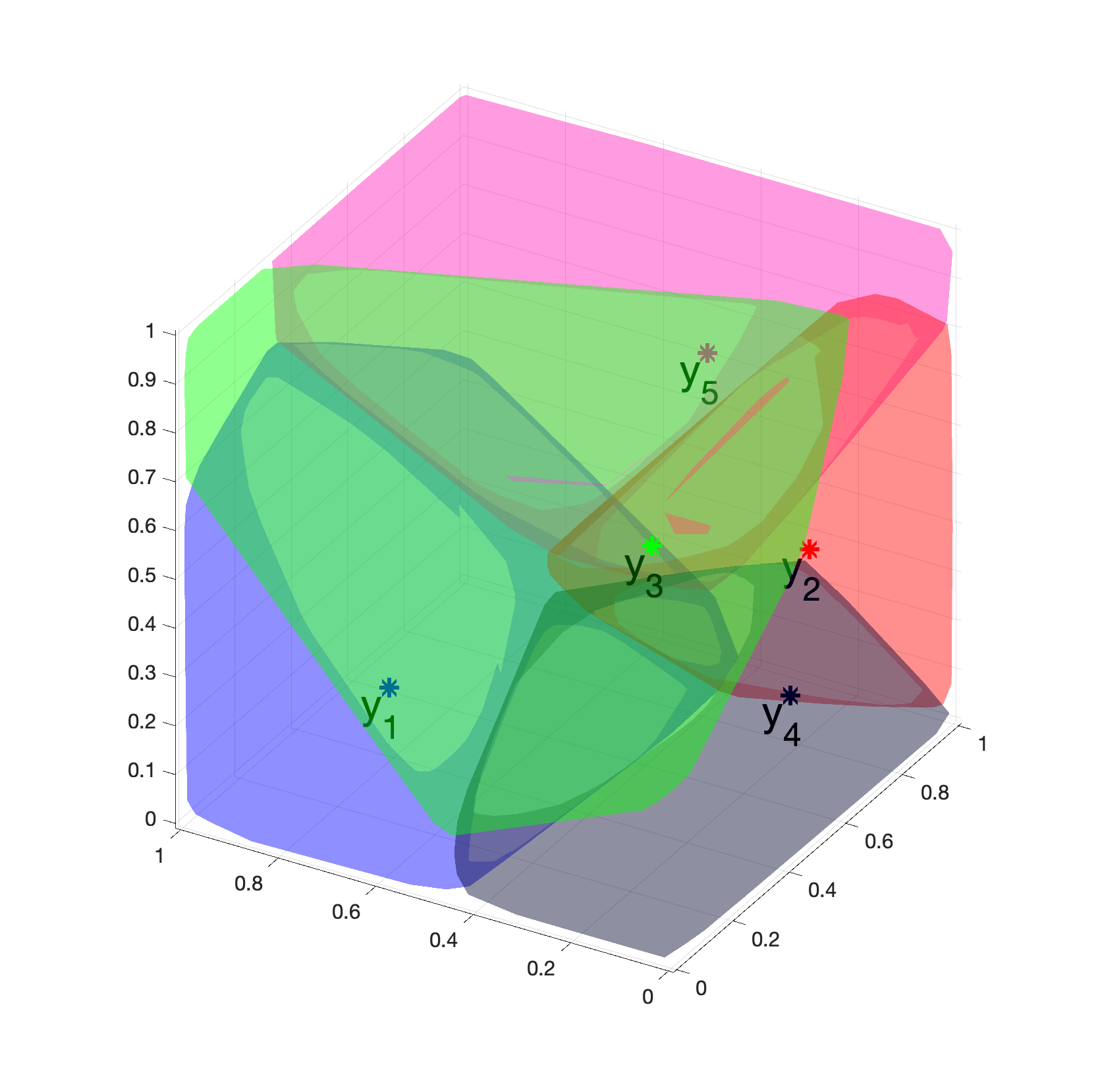}
        \caption{$t=0.25$}
        \label{fig:3d_5pnts_2}
    \end{subfigure}
    \begin{subfigure}{.19\textwidth}
        \centering
        \includegraphics[width=\linewidth]{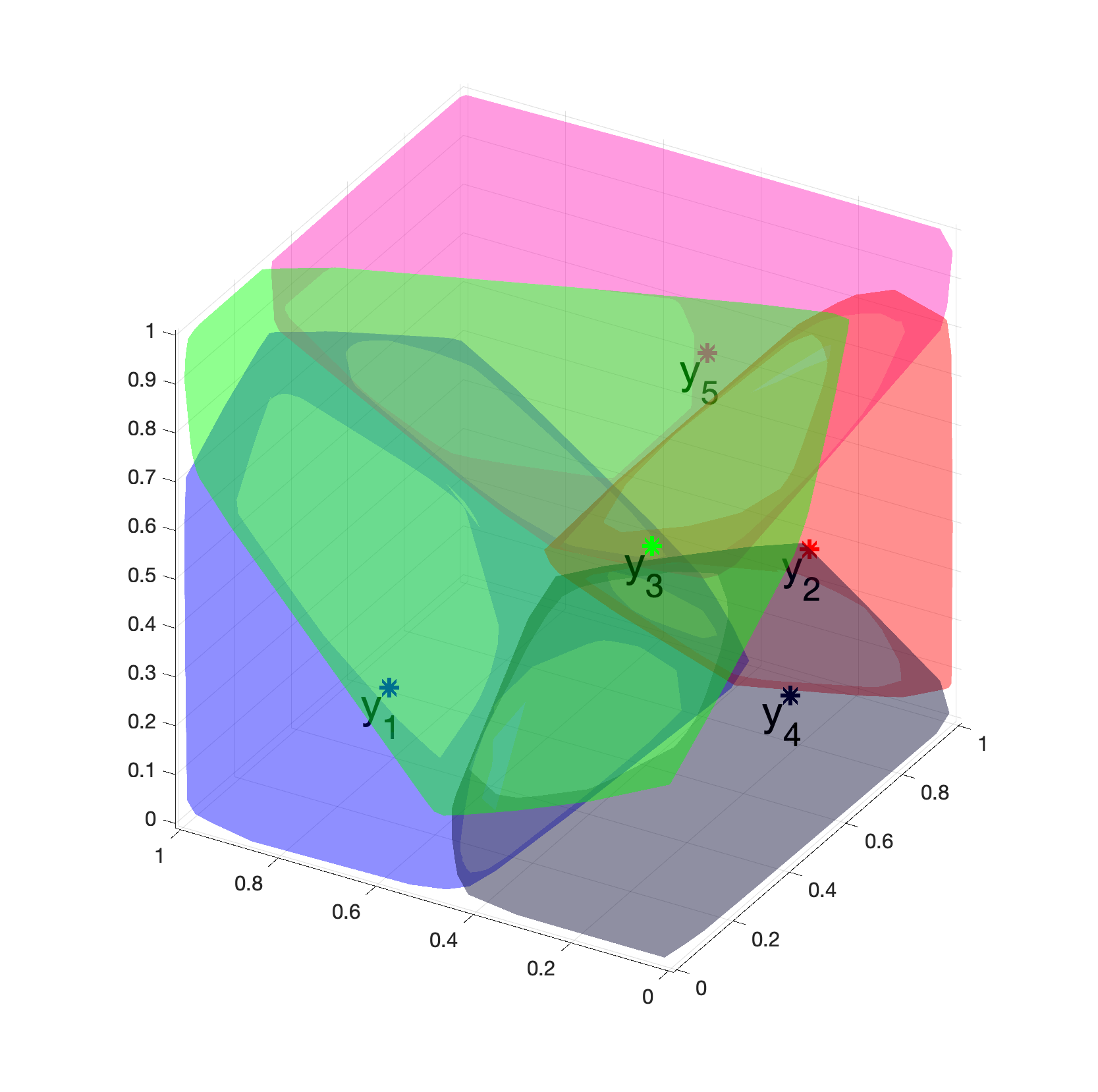}
        \caption{$t=0.5$}
        \label{fig:3d_5pnts_3}
    \end{subfigure}
    \begin{subfigure}{.19\textwidth}
        \centering
        \includegraphics[width=\linewidth]{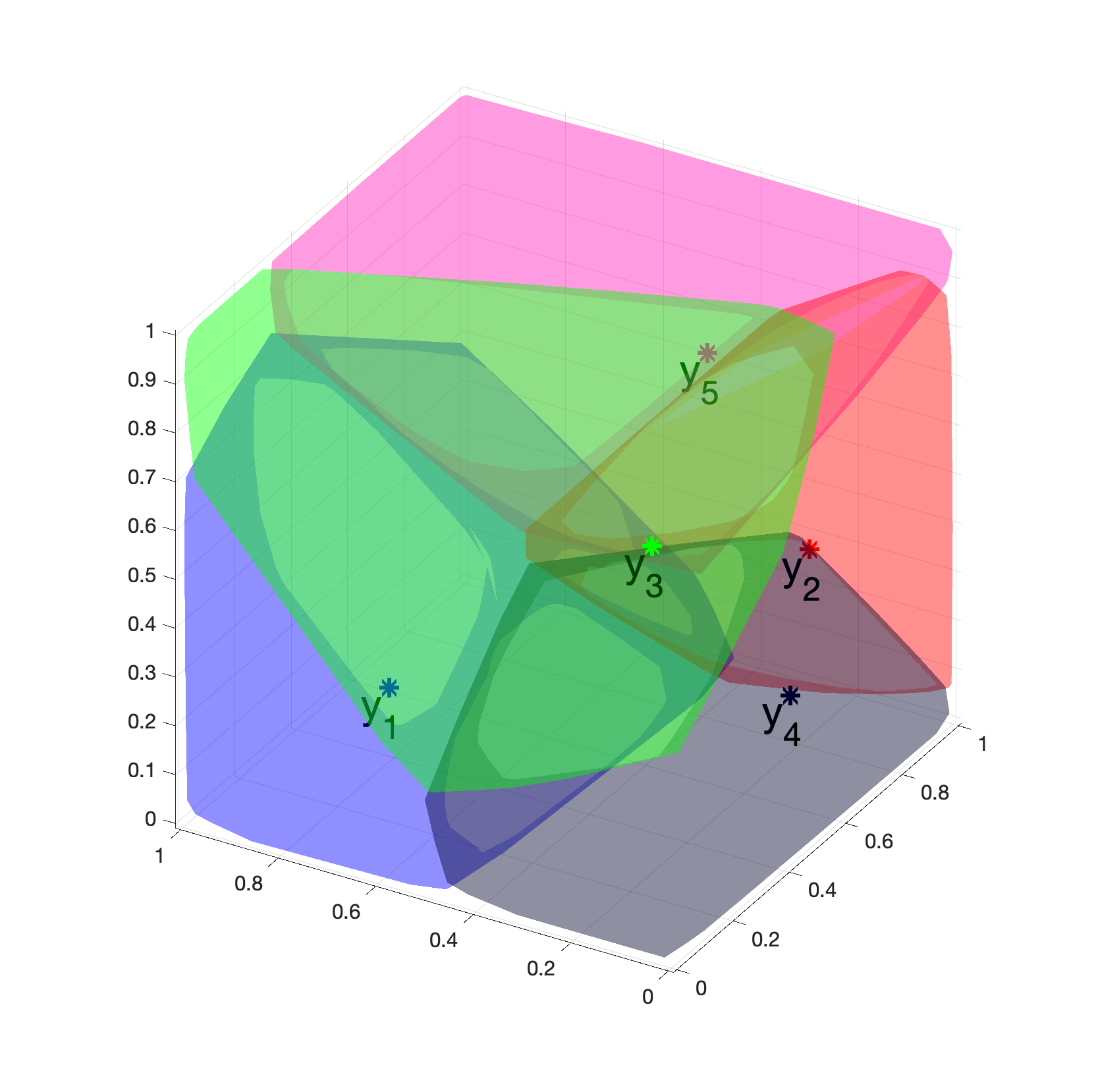}
        \caption{$t=0.75$}
        \label{fig:3d_5pnts_4}
    \end{subfigure}
    \begin{subfigure}{.19\textwidth}
        \centering
        \includegraphics[width=\linewidth]{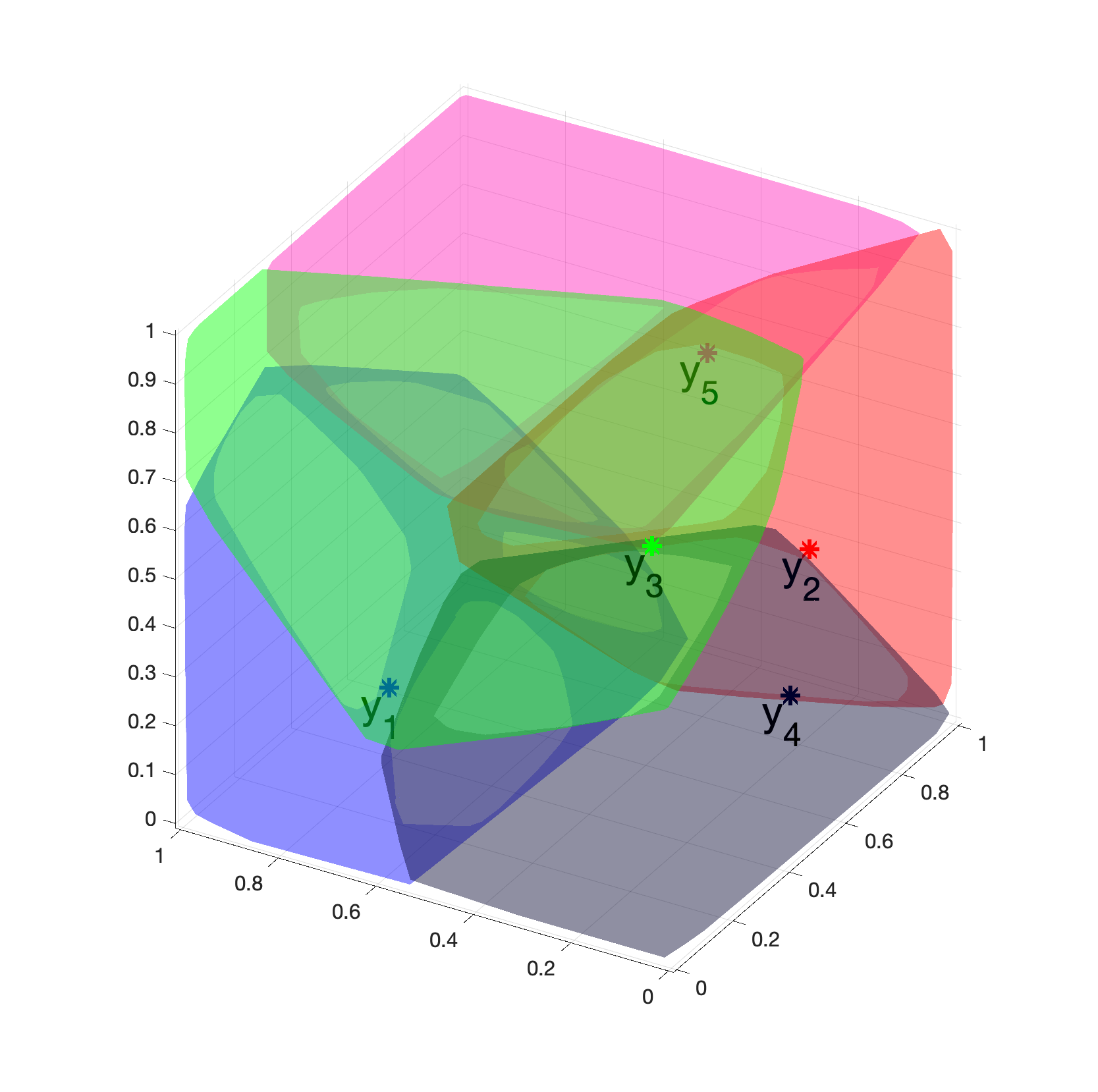}
        \caption{$t=1$}
        \label{fig:3d_5pnts_5}
    \end{subfigure}
    \caption{Time evolution of Laguerre cells in Example \eqref{E7}}
    \label{fig:Lag_3d_E7}
\end{figure}

\begin{figure}[ht]
    \centering
    \includegraphics[width=0.7\linewidth]{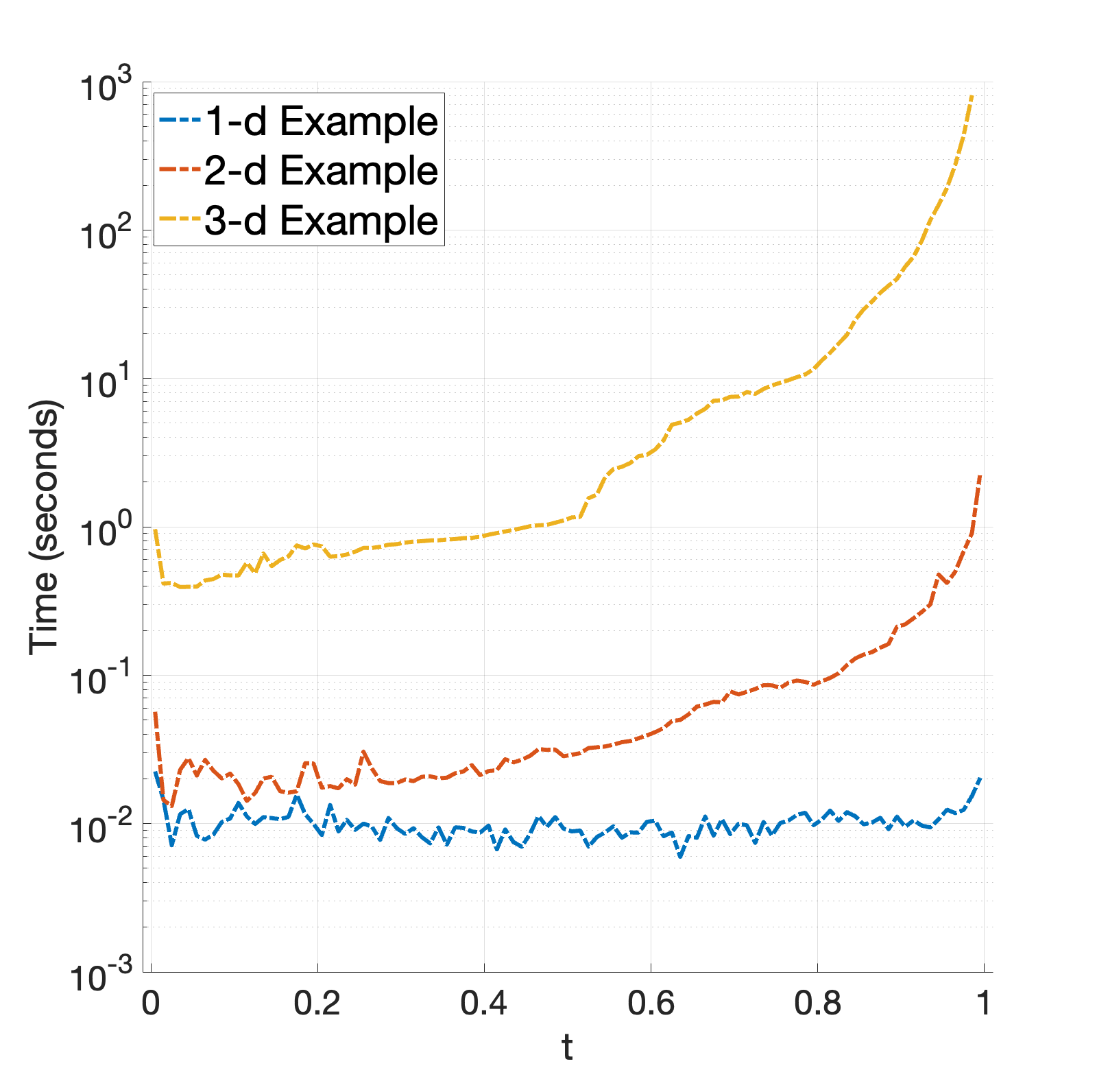}
    \caption{Comparison of time across dimensions and as $t\to 1$}
    \label{fig:Time_Comparison}
\end{figure}

\subsection*{Acknowledgments}
 L.N. benefited from the support of the FMJH Program PGMO and from the ANR project GOTA (ANR-23-CE46-0001).  L.N. thanks T. O. Gallouët, Q. Mérigot and P. Pegon for the fruitful discussions on semi-disrete OT and entropic regularization.
 
 \noindent D.O. gratefully acknowledge that this research was supported in part by the Pacific Institute for the Mathematical Sciences.
 
 \noindent B.P. is pleased to acknowledge the support of Natural Sciences and
Engineering Research Council of Canada Discovery Grant number  04864-2024.
\bibliography{references}{}
\bibliographystyle{plain}

\appendix
\section{Proof of Proposition \ref{prop:derivative}}
\label{appendix}
\begin{proof}
The proof is based on the one of \cite{delalande2022nearly}[Proposition 5.1] adapted to the case of a generic cost function.
First of all let us re-write $\frac{\partial}{\partial t}\nabla_\psi \Phi(\psi,t)$ as follows
\[\frac{\partial}{\partial t}\nabla \Phi(\psi,t)=\int_X\sum_{j=1}^N\bigg(\dfrac{\Delta_{ij}(x,1)}{(1-t)^2}\bigg)\pi_i(x)\pi_j(x)\dd\rho(x)\ , \]
where the quantity
\[\Delta_{i,j}(x,t)=\psi_i -tc(x,y_i)-(\psi_j - tc(x,y_j))\ ,\]
can be interpreted as a duality gap and
\[\pi_i(x)=\dfrac{1}{\sum_{k=1}^N \exp{\bigg(\dfrac{\Delta_{ki}(x,t)}{1-t}}\bigg)}\ .\]
In order to estimate the integral over $X$ we are going to make a similar partition as the one in \cite{delalande2022nearly}: we consider the set
\[X_{i,\eta,+}:=\{x\in\text{Lag}_i(\psi)\;|\;\forall j\neq i,\dfrac{\Delta_{ij}(x,1)}{||p^x_i-p^x_j||}\geq \eta\}\ ,\]
and
\[X_{i,\eta,-}:=\{x\in X\;|\;\forall j\in\text{argmax}_k(\psi_k-c(x,y_k)),\dfrac{\Delta_{ji}(x,1)}{||p^x_j-p^x_i||}\geq \eta\}\ ,\]
where $p^x_i:=\nabla_xc(x,y_i)$. Notice that this two sets corresponds respectively to the points of $\mathrm{Lag}_i(\psi)$ and $X\setminus \mathrm{Lag}_i(\psi)$ which are far from the boundary of the Laguerre cell $i$.
%at a distance, between the vectors $p^x_i$ and $p_j^x$ belonging to tangent space, at $\nabla_x c(x,\cdot)$  at least $\eta$ from the boundary of $\mathrm{Lag}_i(\psi)$.
The idea behind the proof is to create then a tube around the common boundary between the Laguerre cells to have better estimates. Moreover, notice that thanks to the twisted assumption the quantity $||p^x_j-p^x_i||$ is never null. We the define for any $j\neq i$ the common boundary between $\mathrm{Lag}_i(\psi)$ and   $\mathrm{Lag}_j(\psi)$ as $H_{ij}=\mathrm{Lag}_i(\psi)\cap \mathrm{Lag}_j(\psi)$ and for a parameter $\gamma>0$ the set of points of $H_{ij}$  that are at least a distance controlled by $\gamma$ from the other Laguerre cells
\[H^\gamma_{ij}:=\{x_0\in H_{ij}\;|\;\forall k\neq i,j,\Delta_{ik}(x_0,1)=\Delta_{jk}(x_0,1)\geq \gamma\max(||p_i^{x_0}-p_k^{x_0}||,||p_j^{x_0}-p_k^{x_0}||)\}\ .\]
Denote then by
\[T_{i,\eta,\gamma}=\cup_{j\neq i}\{x_0+sd^{x_0}_{ij},x_0\in H^\gamma_{ij},s\in[-\eta||p_i^{x_0}-p_j^{x_0}||,\eta ||p_i^{x_0}-p_j^{x_0}||]\}\ ,\]
with $d^{x_0}_{ij}=\frac{p_i^{x_0}-p_j^{x_0}}{||p_i^{x_0}-p_j^{x_0}||^2}$ the union of tubular sets around the common boundary without the corners and $C_{i,\eta,\gamma}=X\setminus (X_{i,\eta,+}\cup X_{i,\eta,-}\cup T_{i,\eta,\gamma})$ the neighbourhood of the corners of the Laguerre cell.
Before going into the details notice that thanks to the assumptions on the cost and the set $X$ we can easily control the term $p_i^{x}-p_j^{x}$ for all $x$ and couple of indexes $i,j$.
It is then quite easily to observe that on both $X_{i,\eta,+}$ and  $X_{i,\eta,-}$ we get the following control on the integrand
\[\sum_{j=1}^N\bigg(\dfrac{\Delta_{ij}(x,1)}{(1-t)^2}\bigg)\pi_i(x)\pi_j(x)\lesssim\dfrac{e^{-\eta/(1-t)}}{(1-t)^2}\ . \]
We turn now our attention on the evaluation of the integral over $T_{i,\eta,\gamma}$. Notice first that for $x\in T_{i,\eta,\gamma}$ there exists and index $j$ and $x_0\in H^\gamma_{ij}$ such that $x=x_0+sd^{x_0}_{ij}$ such that we get, by a Taylor expansion,
\[
\begin{split}
\Delta_{ij}(x,1)&=\Delta_{ij}(x_0,1)+s\langle\nabla_x\Delta(x_0,1),d^{x_0}_{ij}\rangle+\frac{s^2}{2}\langle\nabla_{xx}^2\Delta(\xi,1)d^{x_0}_{ij},d^{x_0}_{ij}\rangle\\
&=s+\frac{s^2}{2}\langle\nabla_{xx}^2\Delta(\xi,1)d^{x_0}_{ij},d^{x_0}_{ij}\rangle\ ,
\end{split}
\]
where the quadratic term will not play an important role since we will take $s$ depending by $1-t$ and for $t\rightarrow 1$ the first order term will be the dominant one. 
For any $k\neq i,j$ we have by definition of $H^\gamma_{ij}$ and, again by a Taylor expansion,
\[\Delta_{ik}(x,t)\geq \gamma||p^{x_0}_i-p^{x_0}_j||-|s|C\geq \tilde\gamma\ ,\]
and in the same way  we have $\Delta_{i,k}(x,t)\geq\tilde{\gamma}$.
Now the integral on $T_{i,\eta,\gamma}$ can be written as follows
\begin{equation}
    \sum_{j}\int_{x_0\in H_{ij}^\gamma}\int_{0}^{\eta||p^{x_0}_i-p^{x_0}_j||}(g_i(x_0-sd^{x_0}_{ij})+g_i(x_0+sd^{x_0}_{ij}))\dd t\dd \mathcal H^{d-1}(x_0)\ ,
\end{equation}
where $g_i(x)=\sum_{j\neq i}\big(\frac{\Delta_{ij}(x,1)}{(1-t)^2}\big)\pi_i(x)\pi_j(x)\rho(x)$.
Denote by  $x_s=x_0+sd^{x_0}_{ij}$ and $x_{-s}=x_0-sd^{x_0}_{ij}$ then
\[
\begin{split}
g_i(x_{s})&=\bigg(\dfrac{\Delta_{ij}(x_{-s},1)}{(1-t)^2})\bigg)\pi_i(x_{s})\pi_j(x_s)\rho(x_s)+\sum_{k\neq i,j}\bigg(\dfrac{\Delta_{ik}(x_{-s},1)}{(1-t)^2})\bigg)\pi_i(x_{s})\pi_k(x_s)\rho(x_s)\\
&\lesssim \dfrac{s}{(1-t)^2}\pi_i(x_{s})\pi_j(x_s)\rho(x_s)+\dfrac{s^2}{2(1-t)^2}\pi_i(x_{s})\pi_j(x_s)\rho(x_s)+\dfrac{1}{(1-t)^2}e^{-\tilde\gamma/(1-t)}\ ,
\end{split}
\]
where we have used the fact that $\pi_k(x_s)\leq e^{-\tilde\gamma/(1-t)}$ to get an upper bound on the second term of $g$.
In the same way we obtain
\[g(x_{-s})\lesssim \dfrac{-s}{(1-t)^2}\pi_i(x_{-s})\pi_j(x_{-s})\rho(x_{-s})+\dfrac{s^2}{2(1-t)^2}\pi_i(x_{-s})\pi_j(x_{-s})\rho(x_{-s})+\dfrac{1}{(1-t)^2}e^{-\tilde\gamma/(1-t)}\ , \]
thus we get
\begin{equation}
    \begin{split}
        |g_i(x_{s})+g_i(x_{-s})|&\lesssim \dfrac{s}{(1-t)^2}|\pi_i(x_{s})\pi_j(x_s)\rho(x_s)-\pi_i(x_{-s})\pi_j(x_{-s})\rho(x_{-s})|\\
        &+\dfrac{s^2}{2(1-t)^2}|\pi_i(x_{s})\pi_j(x_s)\rho(x_s)+\pi_i(x_{-s})\pi_j(x_{-s})\rho(x_{-s})|\\
        &+\dfrac{1}{(1-t)^2}e^{-\tilde\gamma/(1-t)}\\
        &\lesssim \dfrac{s}{(1-t)^2}|\pi_i(x_{s})\pi_j(x_s)-\pi_i(x_{-s})\pi_j(x_{-s})|\rho(x_s)\\
        &+\dfrac{s}{(1-t)^2}\pi_i(x_{-s})\pi_j(x_{-s})|\rho(x_s)-\rho(x_{-s})|\\
        &+\dfrac{s^2}{2(1-t)^2}|\pi_i(x_{s})\pi_j(x_s)\rho(x_s)+\pi_i(x_{-s})\pi_j(x_{-s})\rho(x_{-s})|\\
        &+\dfrac{1}{(1-t)^2}e^{-\tilde\gamma/(1-t)}\ .
    \end{split}
\end{equation}
First by Hölder continuity of $\rho$ we obtain
\[|\rho(x_s)-\rho(x_{-s})|\leq C||x_{s}-x_{-s}||^\alpha\lesssim s^{\alpha}\ .\]
Consider now the term $\pi_i(x_s)$ and re-write it as
\[\pi_i(x_s)=\dfrac{1}{1+\exp{\bigg(\dfrac{\Delta_{ji}(x,t)}{1-t}\bigg)}+\sum_{k\neq i,j}^N \exp{\bigg(\dfrac{\Delta_{ki}(x,t)}{1-t}}\bigg)}\ ,\]
now for the second term at the denominator we can use the Taylor expansion where as for the third we have the bound $\lesssim e^{-\tilde\gamma/(1-t)}$ so that we get
\[|\pi_i(x_{s})\pi_j(x_s)-\pi_i(x_{-s})\pi_j(x_{-s})|\lesssim e^{-\tilde\gamma/(1-t)}e^{t(s+s^2)/(1-t)}\ . \]
Injecting now all these bounds into the integral on $T_{i,\eta,\gamma}$, and considering that $\eta$ is going to be very small,  yields 
\[
\begin{split}
&\bigg\lvert\int_{T_{i,\eta,\gamma}}\sum_{j\neq i}\dfrac{\Delta_{ij}(x,t)}{(1-t)^2}\pi_i(x)\pi_j(x)\dd\rho(x)\bigg\rvert\lesssim\int_{0}^{\eta\kappa}\dfrac{1}{(1-t)^2}(se^{-\tilde\gamma/(1-t)}e^{ts/(1-t)}+s^2+s^{1+\alpha}+e^{\tilde\gamma/(1-t)})\dd t\\
&\lesssim \dfrac{\eta^{2+\alpha}}{(1-t)^2}+\dfrac{\eta^3}{(1-t)^2}+\dfrac{e^{-\tilde\gamma/(1-t)}}{(1-t)^2}(\eta+(1-t)\eta e^{\eta/(1-t)}-(1-t)^2(e^{\eta/(1-t)}-1))\ ,
\end{split}
\]
where $\kappa=\max_{i\neq j}\max_{x\in X}||p^x_i-p^x_j||$.
The control on $C_{i,\eta,\gamma}$ is exactly obtained as in \cite{delalande2022nearly} and we get that
\[\bigg\lvert\int_{C_{i,\eta,\gamma}}\sum_{j\neq i}\dfrac{\Delta_{ij}(x,t)}{(1-t)^2}\pi_i(x)\pi_j(x)\dd\rho(x)\bigg\rvert\lesssim\dfrac{\gamma^2}{(1-t)^2}(\eta+e^{-\eta/(1-t)})\ .
\]
Finally
\begin{equation*}
\begin{split}
   & \bigg|\int_X\sum_{j=1}^N\bigg(\dfrac{\Delta_{ij}(x,1)}{(1-t)^2}\bigg)\pi_i(x)\pi_j(x)\dd\rho(x)\bigg|\lesssim \bigg|\int_{X_{i,\eta,-}}\sum_{j=1}^N\bigg(\dfrac{\Delta_{ij}(x,1)}{(1-t)^2}\bigg)\pi_i(x)\pi_j(x)\dd\rho(x)\bigg|+\\
   &\bigg|\int_{X_{i,\eta,+}}\sum_{j=1}^N\bigg(\dfrac{\Delta_{ij}(x,1)}{(1-t)^2}\bigg)\pi_i(x)\pi_j(x)\dd\rho(x)\bigg|+\bigg|\int_{T_{i,\eta,\gamma}}\sum_{j=1}^N\bigg(\dfrac{\Delta_{ij}(x,1)}{(1-t)^2}\bigg)\pi_i(x)\pi_j(x)\dd\rho(x)\bigg|+\\
   &\bigg|\int_{C_{i,\eta,\gamma}}\sum_{j=1}^N\bigg(\dfrac{\Delta_{ij}(x,1)}{(1-t)^2}\bigg)\pi_i(x)\pi_j(x)\dd\rho(x)\bigg|\lesssim \dfrac{e^{-\eta/(1-t)}}{(1-t)^2}+\dfrac{\gamma^2}{(1-t)^2}(\eta+e^{-\eta/(1-t)})\\
   &+\dfrac{\eta^{2+\alpha}}{(1-t)^2}+\dfrac{\eta^3}{(1-t)^2}+\dfrac{e^{-\tilde\gamma/(1-t)}}{(1-t)^2}(\eta+(1-t)\eta e^{\eta/(1-t)}-(1-t)^2(e^{\eta/(1-t)}-1))\ ,
   \end{split}
\end{equation*}
and  by correctly choosing $\gamma$ and  $\eta$ we get the result.
\end{proof}

\end{document}